\documentclass[11pt,letterpaper]{article}
\usepackage[english]{babel}
\usepackage{amsmath}
\usepackage{amsfonts, mathrsfs}
\usepackage{amssymb}
\usepackage{verbatim}
\usepackage{graphicx}
\usepackage{color}
\usepackage[numbers]{natbib}
\usepackage{url}
\usepackage{subcaption}

\numberwithin{equation}{section}
\newtheorem{theo}{Theorem}[section]

\newtheorem{lemma}[theo]{Lemma}

\newtheorem{assumption}[theo]{Assumption}
\newtheorem{defn}[theo]{Definition}
\newtheorem{remark}[theo]{Remark}

\newenvironment{proof}[1][Proof]{\textbf{#1.} }{\ \rule{0.5em}{0.5em}}

\newcommand{\bi}{\mathbf{i}}

\allowdisplaybreaks

\author{Mariana Olvera-Cravioto\\ \\ University of North Carolina at Chapel Hill}

\title{Strong couplings for static locally tree-like random graphs}

\date{}

\begin{document}
\maketitle

\begin{abstract}
The goal of this paper is to provide a general purpose result for the coupling of exploration processes of random graphs, both undirected and directed, with their local weak limits when this limit is a marked Galton-Watson process. This class includes in particular the configuration model and the family of inhomogeneous random graphs with rank-1 kernel. Vertices in the graph are allowed to have attributes on a general separable metric space and can potentially influence the construction of the graph itself. The coupling holds for any fixed depth of a breadth-first exploration process.
\vspace{5mm}

\noindent {\em Keywords: } Random graphs, complex networks, Galton-Watson processes, configuration model, inhomogeneous random graph, local-weak limits. 

\end{abstract}

\section{Introduction}

There is a growing literature of problems in physics, mathematics, computer science and operations research that are set up as processes, random or not, on large sparse graphs. The range of problems being studied is wide, and includes problems related to the  classification, sorting, and ranking of large networks, as well as the analysis of Markov chains and interacting particle systems on graphs. Popular among the types of graphs used for these purposes, are the locally tree-like random graph models such as the configuration model and the inhomogeneous random graph family (which includes the classical Erd\H os-R\'enyi model). These random graph models are quite versatile in the types of graphs they can mimic, and have important mathematical properties that make their analysis tractable. 

In particular, the mathematical tractability of locally tree-like random graphs comes from the fact that their local neighborhoods resemble trees. This property makes it easy to transfer questions about the process of interest on a graph, to the often easier analysis of the process on the limiting tree. Mathematically, this transfer is enabled by the notion of local weak convergence \cite{aldous2007processes, aldous2004objective, benjamini2011recurrence, Gar_vdH_Lit_19}. However, as it is the case for many problems involving usual weak convergence of random variables, it is often desirable to construct the original set of random variables and their corresponding weak limits on the same probability space, in other words, to have a coupling. In addition, many problems studying processes on graphs require that we keep track of additional vertex attributes not usually included in the local weak limits, attributes that may not be discrete. The recent work in \cite{fra_lin_olv_21} gives several examples of Markov chains and systems of equations on directed graphs whose analysis relies on the kind of couplings presented here, and include the study of the personalized PageRank distribution \cite{Chen_Lit_Olv_17, Gar_vdH_Lit_19, Olvera_20} among others. Further applications include the analysis of interacting diffusions \cite{lacker2020local, lack_ram_wu_20}, where one may wish to allow the vertex attributes to influence the dynamics of the processes being studied.  The results in this paper were designed to solve these two problems simultaneously, by providing a general purpose coupling between the exploration of the neighborhood of a uniformly chosen vertex in a locally tree-like graph and its local weak limit, including general vertex attributes that may indirectly influence the construction of the graph. 

The main results focus only on the two families of random graph models that are known to converge, in the local weak sense, to a marked Galton-Watson process. It is worth mentioning that other locally tree-like graphs like the preferential attachment models do not fall into this category, since their local weak limits are continuous-time branching processes. In particular, we focus on random graphs constructed according to either a configuration model or any of the inhomogeneous random graph models with rank-1 kernels (see Sections~\ref{SS.CM} and \ref{SS.IR} for the precise definitions). Our results include both undirected and directed graphs, and are given under minimal moment conditions in order to include scale-free graphs. In particular, under our assumptions, it is possible for the offspring distribution in the limiting marked Galton-Watson process to have infinite mean, and in the directed case, for the limiting joint distribution of the in-degree and out degree of a vertex to have infinite covariance. 

Before describing the two families of random graph models for which our coupling theorems hold, we will introduce some definitions that will be used throughout the paper. We will use $G(V_n, E_n)$ to denote a graph, or multigraph, having vertices $V_n = \{1, 2, \dots, n\}$ and edges in the set $E_n$. A directed edge from vertex $i$ to vertex $j$ is denoted by $(i,j)$. For multigraphs, we also need to keep track of the multiplicity of each edge or self loop, so we use $l(i)$ to denote the number of self-loops of vertex $i$ and $e(i,j)$ to denote the number of edges from vertex $i$ to vertex $j$. If the graph is undirected, we simply ignore the direction.  In the undirected case, we use $D_i$ to denote the degree of vertex $i$, which corresponds to the number of adjacent neighbors of vertex $i$. In the directed case, we use $D_i^-$ to denote the in-degree of vertex $i$ and $D_i^+$ to denote its out-degree; the in-degree counts the number of inbound neighbors while the out-degree the number of outbound ones. All our results are given in terms of the large graph limit, which corresponds to taking a sequence of graphs $\{ G(V_n, E_n): n \geq 1 \}$ and taking the limit as $|V_n| = n \to \infty$, where $|A|$ denotes the cardinality of set $A$. Both the configuration model and the family of inhomogeneous random graphs are meant to model large {\em static} graphs, since there may be no relation between $G(V_n, E_n)$ and $G(V_{m}, E_m)$ for $n \geq m$. Strong couplings for {\em evolving} graphs such as the preferential attachment models are a topic for future work.

\subsection{Configuration model} \label{SS.CM}

The configuration model \cite{bollobas, Hofstad1} produces graphs from any prescribed (graphical) degree sequence. In the undirected version of this model, each vertex is assigned a number of stubs or half-edges equal to its target degree. Then, these half-edges are randomly paired to create edges in the graph. 

For an undirected configuration model (CM), we assume that each vertex $i \in V_n$ is assigned an {\em attribute} vector $\mathbf{a}_i = (D_i, \mathbf{b}_i)$, where $D_i \in \mathbb{N}$ is its degree, and $\mathbf{b}_i$ encodes additional information about vertex $i$ that does not directly affect the construction of the graph but may depend on $D_i$. The attributes $\{\mathbf{b}_i\}$ are assumed to take values on a separable metric space $\mathcal{S}'$. For the sequence $\{D_i: 1 \leq i \leq n\}$ to define the degree sequence of an undirected graph, we must have that
$$L_n := \sum_{i=1}^n D_i$$
be even. Note that this may require us to consider a double sequence $\{ {\bf a}_i^{(n)}: i \geq 1, n \geq 1\}$ rather than a unique sequence, i.e., one where $\mathbf{a}_i^{(n)} \neq \mathbf{a}_i^{(m)}$ for $n \neq m$. In applications it is often convenient to allow the vertex attributes to be random themselves; the work in \cite{fra_lin_olv_21} studies various stochastic recursions on graphs that include vertex attributes such as the ones we envision here.

Assuming that $L_n$ is even, enumerate all the stubs, and pick one stub to pair; suppose the stub belongs to vertex $i$. Next, choose one of the remaining $L_n-1$ stubs uniformly at random, and if the stub belongs to vertex $j$, draw an edge between vertices $i$ and $j$; pick another stub to pair. In general, a stub being paired chooses uniformly at random from the set of unpaired stubs, then identifies the vertex to which the chosen stub belongs, and creates an edge between its vertex and the one to which the chosen stub belongs. 

The directed version of the configuration model (DCM) is such that each vertex $i \in V_n$ is assigned an {\em attribute} of the form $\mathbf{a}_i = (D_i^-, D_i^+, \mathbf{b}_i) \in \mathbb{N}^2 \times \mathcal{S}'$. Similarly to the undirected case, $D_i^-$ and $D_i^+$ denote the in-degree and the out-degree, respectively, of vertex $i$, and the $\mathbf{b}_i$ is allowed to depend on $(D_i^-, D_i^+)$. The condition needed to ensure we can draw a graph is now:
$$L_n := \sum_{i=1}^n D_i^+ = \sum_{i=1}^n D_i^-,$$
which again may require us to consider a double sequence $\{ {\bf a}_i^{(n)}: i \geq 1, n \geq 1\}$. 

As for the CM, we give to each vertex $i$ a number $D_i^-$ of inbound stubs, and a number $D_i^+$ of outbound stubs. To construct the graph, we start by choosing an inbound (outbound) stub, say belonging to vertex $i$, and choose uniformly at random one of the $L_n$ outbound (inbound) stubs. If the chosen stub belongs to vertex $j$, draw an edge from $j$ to $i$ (from $i$ to $j$); then pick another inbound (outbound) stub to pair. In general, when pairing an inbound (outbound) stub, we pick uniformly at random from all the remaining unpaired outbound (inbound) stubs. If the stub being paired belongs to vertex $i$, and the one to which the chosen stub belongs to is $j$, we draw a directed edge from $j$ to $i$ (from $i$ to $j$).

We emphasize that both the CM and the DCM are in general multi-graphs, that is, they can have self-loops and multiple edges (in the same direction) between a given pair of vertices. However, provided the pairing process does not create self-loops or multiple edges, the resulting graph is uniformly chosen among all graphs having the prescribed degree sequence.  It is well known that when the empirical degree distribution converges weakly and its second moment converges to that of the limit, the pairing process results in a simple graph with a probability that remains bounded away from zero even as the graph grows \cite{Hofstad1, Chen_Olv_13}. 

We will use $\mathscr{F}_n = \sigma( {\bf a}_i: 1 \leq i \leq n)$ to denote the sigma algebra generated by the attribute sequence, which does not include the edge structure of the graph. To simplify the notation, we will use $\mathbb{P}_n( \cdot ) = P( \cdot | \mathscr{F}_n)$ and $\mathbb{E}_n[ \cdot ] = E[ \cdot | \mathscr{F}_n]$ to denote the conditional probability and conditional expectation, respectively, given $\mathscr{F}_n$.

\subsection{Inhomogeneous random graphs} \label{SS.IR}

The second class of random graph models we consider is the family of inhomogeneous random graphs (digraphs), in which the presence of an edge is determined by the toss of a coin, independently of any other edge. This family includes the classical Erd\H os-R\'enyi graph \cite{Erdos}, but also several generalizations that allow the edge probabilities to depend on the two vertices being connected, e.g.,  the Chung-Lu model \cite{Chunglu}, the Norros-Reittu model (or Poissonian random graph) \cite{Norros}, and the generalized random graph \cite{Brittonetal} to name a few. Unlike the Erd\H os-R\'enyi model, these generalizations are capable of producing graphs with inhomogeneous degree sequences, and can mimic almost any degree distribution whose support is $\mathbb{N}$ (or $\mathbb{N}^2$ in the directed case).  This paper focuses only on inhomogeneous random graphs (digraphs) having rank-1 kernels (see  \cite{Boll_Jan_Rio_07}), which excludes models such as the stochastic block model.

Collectively, this family of models has a long history in the random graph literature, and their connectivity properties, phase transitions, and degree distributions are well known. Rather than attempting to name all the existing references where these models have appeared, we refer the interested reader to the books \cite{Bollobas2, Durrett1, Hofstad1, Hofstad2}, where many of their properties have been compiled.

To define an undirected inhomogeneous random graph (IR), assign to each vertex $i \in V_n$ an {\em attribute} ${\bf a}_i = (W_i, {\bf b}_i)  \in \mathbb{R}_+ \times \mathcal{S}'$. The $W_i$ will be used to determine how likely vertex $i$ is to have neighbors, while the ${\bf b}_i$ can be used to include vertex characteristics that are not needed for the construction of the graph but that are allowed to depend on $W_i$. If convenient, one can consider using a double sequence $\{ \mathbf{a}_i^{(n)}: 1 \geq 1, n \geq 1\}$ as with the configuration model, but this is not as important since the sequence $\mathscr{W}_n := \{W_i: 1 \leq i \leq n\}$ does not need to satisfy any additional conditions in order for us to draw the graph. As with the CM (DCM), the vertex attributes are allowed to be random.

We will use the same notation  $\mathscr{F}_n = \sigma( {\bf a}_i: 1 \leq i \leq n)$, as for the configuration model, to denote the sigma algebra generated by the vertex attributes, as well as the notation for the corresponding conditional probability, $\mathbb{P}_n(\cdot) = P( \cdot | \mathscr{F}_n)$, and expectation, $\mathbb{E}_n[ \cdot] = E[ \cdot | \mathscr{F}_n]$.

For the IR, the edge probabilities are given by:
\begin{equation*}
p_{ij}^{(n)} := \mathbb{P}_n \left( (i,j) \in E_n \right) = 1 \wedge  \frac{W_i W_j}{\theta n} (1 + \varphi_n( W_i, W_j)) , \qquad 1 \leq i < j \leq n,
\end{equation*}
where $-1 < \varphi_n(W_i, W_j) = \varphi(n, W_i, W_j,\mathscr{W}_n)$ a.s.~is a function that may depend on the entire sequence $\mathscr{W}_n$, on the types of the vertices $\{i,j\}$, or exclusively on $n$, and $0<\theta < \infty$ satisfies
$$\frac{1}{n} \sum_{i=1}^n W_i \stackrel{P}{\longrightarrow} \theta, \qquad n \to \infty.$$
 Here and in the sequel, $x \wedge y = \min\{x,y\}$ and $x \vee y = \max\{x, y\}$. Since the graph is to be simple by construction, $p_{ii}^{(n)} \equiv 0$ for all $i \in V_n$. 
 
For the directed version, which we refer to as an inhomogeneous random digraph (IRD), the vertex attributes take the form $\mathbf{a}_i = (W_i^-, W_i^+, \mathbf{b}_i) \in \mathbb{R}_+^2 \times \mathcal{S}'$. The parameter $W_i^-$ controls the in-degree of vertex $i$, and $W_i^+$ its out-degree. If we write $\mathbf{W}_i = (W_i^-, W_i^+)$, the edge probabilities in the IRD are given by:
\begin{equation*}
p_{ij}^{(n)} := \mathbb{P}_n \left( (i,j) \in E_n \right) = 1 \wedge  \frac{W^+_i W^-_j}{\theta n} (1 + \varphi_n( \mathbf{W}_i, \mathbf{W}_j)) , \qquad 1 \leq i \neq j \leq n,
\end{equation*}
where $-1 < \varphi_n(\mathbf{W}_i, \mathbf{W}_j) = \varphi(n, \mathbf{W}_i, \mathbf{W}_j,\mathscr{W}_n)$ a.s.~is a function that may depend on the entire sequence $\mathscr{W}_n:= \{ \mathbf{W}_i: 1 \leq i \leq n\}$, on the types of the vertices $\{i,j\}$, or exclusively on $n$, and $0<\theta < \infty$ satisfies
$$\frac{1}{n} \sum_{i=1}^n (W_i^- + W_i^+) \stackrel{P}{\longrightarrow} \theta, \qquad n \to \infty.$$
Since the graphs are again simple by construction, we have $p_{ii}^{(n)} \equiv 0$ for all $i \in V_n$.

\section{Main result for undirected graphs}

For an undirected graph constructed according to one of the two models (CM or IR), our main result shows that there exists a coupling between the breadth-first exploration of the component of a uniformly chosen vertex and that of the root node of a marked Galton-Watson process. Before we can state the theorem, we need to introduce some notation on the graph and describe the Galton-Watson process that describes its local weak limit. 

Each vertex $i$ in an undirected graph (multigraph) $G(V_n, E_n)$ is given a vertex {\em attribute} of the form:
$${\bf a}_i = \begin{cases} (D_i, {\bf b}_i) & \text{if $G(V_n, E_n)$ is a CM,} \\
(W_i, {\bf b}_i) & \text{if $G(V_n, E_n)$ is an IR.} \end{cases}$$
In addition, define for each vertex $i$ its {\em full mark}:
$${\bf X}_i = (D_i, {\bf a}_i),$$
where $D_i$ is the degree of vertex $i$. We point out that the definition of $\mathbf{X}_i$ is redundant when the graph is a CM, however, it is not so if the graph is an IR. In both cases the vertex attributes are measurable with respect to $\mathscr{F}_n$, while the full marks are not if the graph is an IR.  

The main assumption needed for the coupling to hold is given in terms of the empirical measure for the vertex attributes, i.e.,  
\begin{equation} \label{eq:AttributeDistr}
\nu_n(\cdot ) = \frac{1}{n} \sum_{i=1}^n 1( {\bf a}_i \in \cdot ) .
\end{equation}
In order to state the assumption, recall that the state space for the vertex attributes, $\mathcal{S}'$, is assumed to be a separable metric space under metric $\rho'$. Now define the metric
$$\rho({\bf x}, {\bf y}) = |x_1 - y_1| + |x_2 - y_2| + \rho'({\bf x}_3, {\bf y}_3), \qquad {\bf x} = (x_1, x_2, {\bf x}_3), \, {\bf y} = (y_1, y_2, {\bf y}_3),$$
on the space $\mathcal{S} := \mathbb{N} \times \mathbb{R} \times \mathcal{S}'$, which makes $\mathcal{S}$ a separable metric space as well. Using $\rho$, and for any probability measures $\nu_n, \mu_n$ on the conditional probability space $(\mathcal{S}, \mathscr{F}_n, \mathbb{P}_n)$, define the Wasserstein metric of order one
$$W_1(\nu_n, \mu_n) = \inf\left\{  \mathbb{E}_n\left[ \rho(\mathbf{\hat Y}, \boldsymbol{Y}) \right]: \text{law}(\mathbf{\hat Y} | \mathscr{F}_n) = \nu_n, \, \text{law}(\boldsymbol{Y} | \mathscr{F}_n) = \mu_n \right\}.$$

\begin{assumption} \label{A.PrimitivesU}
\textbf{(Undirected)} Let $\nu_n$ be defined according to \eqref{eq:AttributeDistr}, and suppose there exists a probability measure $\nu$ such that
$$ W_1(\nu_n, \nu) \stackrel{P}{\longrightarrow} 0, \qquad n \to \infty.$$
In addition, assume that the following conditions hold:
\begin{itemize}
\item[A.] In the CM, let $(\mathscr{D}, \boldsymbol{B} )$ be distributed according to $\nu$, and suppose there exists a non-random ${\bf b}_0 \in \mathcal{S}'$ such that $E[\mathscr{D} + \rho'(\boldsymbol{B}, {\bf b}_0) ] < \infty$.

\item[B.] In the IR, let $(W, \boldsymbol{B})$ be distributed according to $\nu$, and suppose the following hold:
\begin{enumerate}
\item[1.] $\displaystyle \mathcal{E}_n = \frac{1}{n} \sum_{i=1}^n \sum_{1 \leq i \neq j\leq n,} |p_{ij}^{(n)} - (r_{ij}^{(n)} \wedge 1) | \xrightarrow{P} 0$ as $n \to \infty$, where $r_{ij}^{(n)} = W_i W_j/(\theta n)$.

\item[2.] There exists a non-random ${\bf b}_0 \in \mathcal{S}'$ such that $E[W + \rho'(\boldsymbol{B}, {\bf b}_0) ] < \infty$.

\end{enumerate}

\end{itemize}
\end{assumption}

Now that we have stated the assumptions for our theorem, we need to describe the local neighborhood of a vertex in the graph $G(V_n, E_n)$. To do this, let $I \in V_n$ denote  a uniformly chosen vertex in $G(V_n, E_n)$; vertices are identified with their  labels in $\{1, 2, \dots, n\}$. Define $A_0 = \{ I \}$, and let $A_k$ denote the set of vertices at hop distance $k$ from $I$. Now write $\mathcal{G}_I^{(k)}$ to be the subgraph of $G(V_n, E_n)$ consisting of the vertices in $\bigcup_{r=0}^k A_r$ along with their (multiple) edges and self-loops. We will also use the notation $\mathcal{G}_I^{(k)}(\mathbf{a})$ to refer to the graph $\mathcal{G}_I^{(k)}$ including all the attributes of its vertices. 

\begin{defn}
We say that two simple graphs $G(V, E)$ and $G'(V', E')$ are {\em isomorphic} if there exists a bijection $\sigma: V \to V'$ such that edge $(i,j) \in E$ if and only if edge $(\sigma(i), \sigma(j)) \in E'$.  We say that two multigraphs $G(V,E)$ and $G'(V', E')$ are {\em isomorphic} if there exists a bijection $\sigma: V \to V'$ such that $l(i) = l(\sigma(i))$ and $e(i,j) = e(\sigma(i), \sigma(j))$ for all $i \in V$ and all $(i,j) \in E$, where $l(i)$ is the number of self-loops of vertex $i$ and $e(i,j)$ is the number of edges from vertex $i$ to vertex $j$. In both cases, we write $G \simeq G'$.
\end{defn}

To describe the limit of $\mathcal{G}_I^{(k)}$ as $n \to \infty$, we will construct a delayed marked Galton-Watson process, denoted $\mathcal{T}(\boldsymbol{A})$, using the measure $\nu$ in Assumption~\ref{A.PrimitivesU}. The ``delayed" refers to the fact that the root will, in general, have a different distribution than all other nodes in the tree. 

To start, let $\mathcal{U} := \bigcup_{k=0}^\infty \mathbb{N}_+^k$ denote the set of labels for nodes in a tree, with the convention that $\mathbb{N}_+^0 := \{ \emptyset \}$ contains the root. For a label $\mathbf{i} = (i_1, \dots, i_k)$ we write $|\mathbf{i}| = k$ to denote its length, and use $(\mathbf{i}, j) = (i_1, \dots, i_k, j)$ to denote the index concatenation operation. 

The tree $\mathcal{T}$ is constructed as follows. Let $\{ (\mathcal{N}_\mathbf{i}, \boldsymbol{A}_\mathbf{i}): \mathbf{i} \in \mathcal{U} \}$ denote a sequence of independent vectors in $\mathcal{S}$, with $\{ (\mathcal{N}_\mathbf{i}, \boldsymbol{A}_\mathbf{i}): \mathbf{i} \in \mathcal{U}, \mathbf{i} \neq \emptyset \}$ i.i.d. For any $\mathbf{i} \in \mathcal{U}$, the $\mathcal{N}_\mathbf{i}$ will denote the number of offspring of node $\mathbf{i}$, and $\boldsymbol{A}_\mathbf{i}$ will denote its attribute (mark). 
As with the graph, we will use the notation $\mathcal{T}$ to denote the tree without its attributes. Let $\mathcal{A}_0 = \{ \emptyset \}$ and recursively define 
$$\mathcal{A}_k = \{(\mathbf{i}, j):  \mathbf{i} \in \mathcal{A}_{k-1} , \, 1 \leq j \leq \mathcal{N}_\mathbf{i} \}, \qquad k \geq 1,$$
to be the $k$th generation of $\mathcal{T}$. To match the notation on the graph, we write
$$\boldsymbol{X}_\emptyset = (\mathcal{N}_\emptyset, \boldsymbol{A}_\emptyset), \quad \text{and} \quad \boldsymbol{X}_\bi = (\mathcal{N}_\bi+1, \boldsymbol{A}_\bi), \quad \bi \neq \emptyset.$$
The marked tree is then given by $\mathcal{T}(\boldsymbol{A}) = \{ \boldsymbol{X}_\bi: \bi \in \mathcal{T} \}$; note that the marks include the number of offspring of each node, from where the edges in the tree can be deduced. We will denote $\mathcal{T}^{(k)}$ ($ \mathcal{T}^{(k)}(\boldsymbol{A})$) to be the restriction of $\mathcal{T}$ ($\mathcal{T}(\boldsymbol{A})$) to its first $k$ generations.

It only remains to identify the distribution of $\boldsymbol{X}_\bi$, for both $\bi = \emptyset$ and $\bi \neq \emptyset$, in terms of the probability measure $\nu$ in Assumption~\ref{A.PrimitivesU}. For a CM, let $\boldsymbol{A} = (\mathscr{D}, \boldsymbol{B})$ be distributed according to $\nu$, then, 
\begin{align*}
P(\boldsymbol{X}_\emptyset \in \cdot ) &= P( (\mathscr{D}, \boldsymbol{A}) \in \cdot), \\
P(\boldsymbol{X}_\bi \in \cdot) &= \frac{1}{E[\mathscr{D}]} E\left[ \mathscr{D}1( (\mathscr{D}, \boldsymbol{A}) \in \cdot) \right], \quad \bi \neq \emptyset. 
\end{align*}
For an IR, let $\boldsymbol{A} = (W, \boldsymbol{B})$ be distributed according to $\nu$, then,
\begin{align*}
P(\boldsymbol{X}_\emptyset \in \cdot ) &= P( (D, \boldsymbol{A}) \in \cdot), \\
P(\boldsymbol{X}_\bi \in \cdot ) &= \frac{1}{E[W]} E\left[ W 1( (D+1, \boldsymbol{A}) \in \cdot) \right], \quad \bi \neq \emptyset,
\end{align*}
where $D$ is a mixed Poisson random variable with mean $W$. Note that the distribution of $\boldsymbol{X}_\bi$ for $\bi \neq \emptyset$, corresponds to a {\em size-biased} version of the distribution of $\boldsymbol{X}_\emptyset$ with respect to its first coordinate.

We are now ready to state the main coupling theorem for undirected graphs. 

\begin{theo} \label{T.MainU}
Suppose $G(V_n, E_n)$ is either a CM or an IR satisfying Assumption~\ref{A.PrimitivesU}. Then, for any fixed $k$ and $\mathcal{G}^{(k)}_I({\bf a})$ the depth-$k$ neighborhood of a uniformly chosen vertex $I \in V_n$, there exists a marked Galton-Watson tree $\mathcal{T}^{(k)} (\boldsymbol{A})$ restricted to its first $k$ generations, whose root corresponds to vertex $I$, and 
such that,
$$\mathbb{P}_n\left( \mathcal{G}_I^{(k)} \not\simeq \mathcal{T}^{(k)} \right) \xrightarrow{P} 0, \qquad n \to \infty,$$
and if we let $\sigma({\bf i}) \in V_n$ denote the vertex in the graph corresponding to node ${\bf i} \in \mathcal{T}^{(k)}$, and define for any $\epsilon > 0$ the event
$$C_{I}^{(k,\epsilon)} = \left\{   \bigcap_{{\bf i} \in \mathcal{T}^{(k)}}   \{ \rho(\mathbf{X}_{\sigma({\bf i})}, \boldsymbol{X}_{\bf i} ) \leq \epsilon \}   , \,  \mathcal{G}_I^{(k)} \simeq \mathcal{T}^{(k)} \right\},$$
then
$$\mathbb{E}_n\left[ \rho(\mathbf{X}_I, \boldsymbol{X}_\emptyset) \right] \xrightarrow{P} 0 \qquad \text{and} \qquad  \mathbb{P}_n \left( C_I^{(k,\epsilon)}  \right) \xrightarrow{P} 1, \qquad n \to \infty.$$
Moreover, for any fixed $m, k \geq 1$, $\epsilon > 0$, and $\{I_j: 1 \leq j \leq m\}$ i.i.d.~random variables uniformly chosen in $V_n$, there exist i.i.d.~copies of $\mathcal{T}^{(k)}(\boldsymbol{A})$, denoted $\{ \mathcal{T}_{\emptyset(I_j)}^{(k)}(\boldsymbol{A}): 1 \leq j \leq m\}$, whose roots correspond to the vertices $\{I_j: 1 \leq j \leq m\}$ in $G(V_n, E_n)$, such that
$$\sum_{j=1}^m \mathbb{E}_n\left[ \rho(\mathbf{X}_{I_j}, \boldsymbol{X}_{\emptyset(I_j)}) \right] \xrightarrow{P} 0 \qquad \text{and} \qquad \mathbb{P}_n\left( \bigcap_{j=1}^m C_{I_j}^{(k,\epsilon)} \right) \xrightarrow{P} 1, \qquad n \to \infty.$$
\end{theo}

\begin{remark}
Theorem~\ref{T.MainU} implies that the graph $G(V_n, E_n)$ converges in the local weak sense in probability, as introduced in \cite{aldous2007processes, aldous2004objective, benjamini2011recurrence} for undirected graphs and later extended in \cite{Gar_vdH_Lit_19} to marked undirected graphs. The statement involving more than one exploration is related to the notion of propagation of chaos in the interacting particles literature (see, for example, Definition~4.1 in \cite{Chain_Diez_21}), and can be of independent interest in that context. 
\end{remark}

\section{Main result for directed graphs}

In the directed case, our main result will allow us to couple the breadth-first exploration of either the in-component or the out-component of a uniformly chosen vertex. Since the two cases are clearly symmetric, we state our results only for the in-component. 

As with the undirected graph, each vertex $i$ in the graph $G(V_n, E_n)$ has an {\em attribute}:
$$\mathbf{a}_i = \begin{cases} (D_i^-, D_i^+, \mathbf{b}_i) & \text{if $G(V_n, E_n)$ is a DCM}, \\
(W_i^-, W_i^+, \mathbf{b}_i) & \text{if $G(V_n, E_n)$ is an IRD}. \end{cases}$$
The {\em full mark} of vertex $i$ is now given by:
$$\mathbf{X}_i = (D_i^-, D_i^+, \mathbf{a}_i),$$
where $D_i^-$ and $D_i^+$ are the in-degree and out-degree, respectively, of vertex $i$. 

With some abuse of notation, we use again $\nu_n$, as defined in \eqref{eq:AttributeDistr}, to denote the empirical measure for the vertex attributes. However, the state space for the full marks is now $\mathcal{S} := \mathbb{N}^2 \times \mathbb{R}^2 \times \mathcal{S}'$, equipped with the metric:
$$\rho(\mathbf{x}, \mathbf{y}) = |x_1 - y_1| + |x_2 - y_2| + |x_3 - y_3| + |x_4 - y_4| + \rho'(\mathbf{x}_5, \mathbf{y}_5),$$
for $\mathbf{x} = (x_1, x_2, x_3, x_4, \mathbf{x}_5)$ and $\mathbf{y} = (y_1, y_2, y_3, y_4, \mathbf{y}_5)$. The Wasserstein metric $W_1$ defined on the conditional probability space $(\mathcal{S}, \mathscr{F}_n, \mathbb{P}_n)$ remains the same after the adjustments made to $\mathcal{S}$ and $\rho$.

\begin{assumption} \label{A.PrimitivesD}
\textbf{(Directed)} Let $\nu_n$ be defined according to \eqref{eq:AttributeDistr}, and suppose there exists a probability measure $\nu$  such that
$$ W_1(\nu_n, \nu) \stackrel{P}{\longrightarrow} 0, \qquad n \to \infty.$$
In addition, assume that the following conditions hold:
\begin{itemize}
\item[A.] In the DCM, let $(\mathscr{D}^-, \mathscr{D}^+, \boldsymbol{B} )$ be distributed according to $\nu$, and suppose there exists a non-random ${\bf b}_0 \in \mathcal{S}'$ such that $E[\mathscr{D}^- + \mathscr{D}^+ + \rho(\boldsymbol{B}, {\bf b}_0) ] < \infty$.

\item[B.] In the IRD, let $(W^-, W^+, \boldsymbol{B})$ be distributed according to $\nu$, and suppose the following hold:
\begin{enumerate}
\item[1.] $\displaystyle \mathcal{E}_n = \frac{1}{n} \sum_{i=1}^n \sum_{1 \leq i \neq j\leq n,} |p_{ij}^{(n)} - (r_{ij}^{(n)} \wedge 1) | \xrightarrow{P} 0$ as $n \to \infty$, where $r_{ij}^{(n)} = W_i^+ W_j^-/(\theta n)$.

\item[2.] There exists a non-random ${\bf b}_0 \in \mathcal{S}'$ such that $E[W^- + W^+ + \rho(\boldsymbol{B}, {\bf b}_0) ] < \infty$.

\end{enumerate}

\end{itemize}
\end{assumption}

Since we will state our result for the exploration of the in-component of a uniformly chosen vertex, the structure of the coupled tree will be determined by the vertices that we encounter during a breadth-first exploration. This exploration starts with a uniformly chosen vertex $I \in V_n$, which is used to create the set $A_0 = \{ I\}$. It then follows all the inbound edges of $I$ to discover all the vertices at inbound distance one from $I$, which become the set $A_1$. In general, to identify the vertices in the set $A_k$, we explore all the inbound edges of vertices in $A_{k-1}$. As we perform the exploration, we also discover the out-degrees of the vertices we have encountered, however, we do not follow any outbound edges. We then define $\mathcal{G}_I^{(k)}$ to be the subgraph of $G(V_n, E_n)$ whose vertex set is $\bigcup_{r=0}^k A_r$ and whose edges are those that are encountered during the breadth-first exploration we described. The notation $\mathcal{G}_I^{(k)}(\mathbf{a})$ will be used to refer to the graph $\mathcal{G}_I^{(k)}$ including the values of the full marks $\{ \mathbf{X}_i\}$ for all of its vertices. 

In the directed case, the limit of $\mathcal{G}_I^{(k)}$ is again a delayed marked Galton-Watson process, with the convention that all its  edges are pointing towards the root. We will denote the tree $\mathcal{T}(\boldsymbol{A})$ as before, however, it will be constructed using a sequence of independent vectors of the form $\{ (\mathcal{N}_\bi, \mathcal{D}_\bi, \boldsymbol{A}_\bi): \bi \in \mathcal{U}\}$, with $\{ (\mathcal{N}_\bi, \mathcal{D}_\bi, \boldsymbol{A}_\bi): \bi \in \mathcal{U}, \bi \neq \emptyset\}$ i.i.d. In other words, the full marks now take the form
$$\boldsymbol{X}_\bi = (\mathcal{N}_\bi, \mathcal{D}_\bi, \boldsymbol{A}_\bi), \quad \bi \in \mathcal{U}. $$
The construction of the tree $\mathcal{T}$ is done as in the undirected case using the $\{\mathcal{N}_\bi: \bi \in \mathcal{U} \}$, and the marked tree is given by $\mathcal{T}(\boldsymbol{A}) = \{ \boldsymbol{X}_\bi : \bi \in \mathcal{T} \}$. The notation $\mathcal{T}^{(k)}$ ($\mathcal{T}^{(k)}(\boldsymbol{A})$) refers again to the restriction of $\mathcal{T}$ ($\mathcal{T}(\boldsymbol{A})$) to its first $k$ generations. 

The distribution of the full marks $\mathbf{X}_\bi$ for both $\bi = \emptyset$ and $\bi \neq \emptyset$ are also different than in the undirected case. For a DCM, let $\boldsymbol{A} = (\mathscr{D}^-, \mathscr{D}^+, \boldsymbol{B})$ be distributed according to $\nu$, then
\begin{align*}
P(\boldsymbol{X}_\emptyset \in \cdot) &= P( (\mathscr{D}^-, \mathscr{D}^+, \boldsymbol{A}) \in \cdot), \\
P(\boldsymbol{X}_\bi \in \cdot) &= \frac{1}{E[\mathscr{D}^+]} E\left[ \mathscr{D}^+ 1((\mathscr{D}^-, \mathscr{D}^+, \boldsymbol{A}) \in \cdot) \right], \quad \bi \neq \emptyset. 
\end{align*}
For an IRD, let $\boldsymbol{A} = (W^-, W^+, \boldsymbol{B})$ be distributed according to $\nu$, then
\begin{align*}
P(\boldsymbol{X}_\emptyset \in \cdot) &= P( (D^-, D^+, \boldsymbol{A}) \in \cdot), \\
P(\boldsymbol{X}_\bi \in \cdot) &= \frac{1}{E[W^+]} E\left[ W^+ 1((D^-, D^++1, \boldsymbol{A}) \in \cdot) \right], \quad \bi \neq \emptyset,
\end{align*}
where $D^-$ and $D^+$ are conditionally independent (given $(W^-, W^+)$) Poisson random variables with means $c W^-$ and $(1-c)W^+$, respectively, and $c = E[W^+]/E[W^- + W^+]$. Note that in this case, the distribution of $\boldsymbol{X}_\bi$ for $\bi \neq \emptyset$, corresponds to a {\em size-biased} version of the distribution of $\boldsymbol{X}_\emptyset$ with respect to its second coordinate. 

The following is our main coupling theorem for directed graphs.

\begin{theo} \label{T.MainD}
Suppose $G(V_n, E_n)$ is either a DCM or an IRD satisfying Assumption~\ref{A.PrimitivesD}. Then, for any fixed $k$ and $\mathcal{G}^{(k)}_I({\bf a})$ the depth-$k$ neighborhood of a uniformly chosen vertex $I \in V_n$, there exists a marked Galton-Watson tree $\mathcal{T}^{(k)} (\boldsymbol{A})$ restricted to its first $k$ generations, whose root corresponds to vertex $I$, and 
such that,
$$\mathbb{P}_n\left( \mathcal{G}_I^{(k)} \not\simeq \mathcal{T}^{(k)} \right) \xrightarrow{P} 0, \qquad n \to \infty,$$
and if we let $\sigma({\bf i}) \in V_n$ denote the vertex in the graph corresponding to node ${\bf i} \in \mathcal{T}^{(k)}$, and define for any $\epsilon > 0$ the event
$$C_{I}^{(k,\epsilon)} = \left\{  \bigcap_{{\bf i} \in \mathcal{T}^{(k)}}   \{ \rho(\mathbf{X}_{\sigma({\bf i})}, \boldsymbol{X}_{\bf i} ) \leq \epsilon \}   , \,  \mathcal{G}_I^{(k)} \simeq \mathcal{T}^{(k)} \right\},$$
then
$$\mathbb{E}_n\left[ \rho(\mathbf{X}_I, \boldsymbol{X}_\emptyset) \right] \xrightarrow{P} 0 \qquad \text{and} \qquad  \mathbb{P}_n \left(  C_I^{(k,\epsilon)}  \right) \xrightarrow{P} 1, \qquad n \to \infty.$$
Moreover, for any fixed $m, k \geq 1$, $\epsilon > 0$, and $\{I_j: 1 \leq j \leq m\}$ i.i.d.~random variables uniformly chosen in $V_n$, there exist i.i.d.~copies of $\mathcal{T}^{(k)}(\boldsymbol{A})$, denoted $\{ \mathcal{T}_{\emptyset(I_j)}^{(k)}(\boldsymbol{A}): 1 \leq j \leq m\}$, whose roots correspond to the vertices $\{I_j: 1 \leq j \leq m\}$ in $G(V_n, E_n)$, such that
$$\sum_{j=1}^m \mathbb{E}_n\left[ \rho(\mathbf{X}_{I_j}, \boldsymbol{X}_{\emptyset(I_j)}) \right] \xrightarrow{P} 0 \qquad \text{and} \qquad \mathbb{P}_n\left( \bigcap_{j=1}^m C_{I_j}^{(k,\epsilon)} \right) \xrightarrow{P} 1, \qquad n \to \infty.$$
\end{theo}

The remainder of the paper contains the proofs of Theorem~\ref{T.MainU} and Theorem~\ref{T.MainD}.

\section{Proofs}

The proofs of Theorem~\ref{T.MainU} and Theorem~\ref{T.MainD} are based on an intermediate coupling between the breadth-first exploration of the graph $G(V_n, E_n)$ and a delayed marked Galton-Watson process whose offspring distribution and marks still depend on the filtration $\mathscr{F}_n$. This intermediate step consists in coupling $\mathcal{G}^{(k)}_I$ with a marked tree denoted $\hat T^{(k)}(\mathbf{\hat A})$. Interestingly, this coupling will be perfect, in the sense that the vertex/node marks in each of the two graphs will also be identical to each other. The proofs of Theorems~\ref{T.MainU} and \ref{T.MainD} will be complete once we show that $\hat T^{(k)}(\mathbf{\hat A})$ can be coupled with the limiting  $\mathcal{T}^{(k)}(\boldsymbol{A})$.

To organize the exposition, we will separate the undirected case from the directed one. Most of the proofs for the directed case have been established in prior work by the author, and therefore will be omitted; the precise references and the missing details are given in Section~\ref{SS.CouplingDirected}. Once the intermediate coupling theorems are proved, the coupling between the two trees can be done indistinctly for the undirected and directed cases (on the trees, the direction of the edges is irrelevant). 

\subsection{Discrete coupling for undirected graphs}

As mentioned above, the main difference between the intermediate tree and the limiting one lies on the distribution of the marks. As before, we start with the construction of the possibly infinite tree $\hat T$, which is done with the conditionally independent (given $\mathscr{F}_n$) sequence of random vectors in $\mathcal{S}$, $\{ (\hat N_\bi, \mathbf{\hat A}_\bi): \bi \in \mathcal{U}\}$, with $\{ (\hat N_\bi, \mathbf{\hat A}_\bi): \bi \in \mathcal{U}, \bi \neq \emptyset \}$ conditionally i.i.d. Let $\hat A_0 = \{ \emptyset\}$ and recursively define 
$$\hat A_k = \{ (\bi, j): \bi \in \hat A_{k-1}, 1 \leq j \leq \hat N_\bi \}, \qquad k \geq 1.$$
Next, define the full marks according to:
$$\mathbf{\hat X}_\emptyset = (\hat N_\emptyset, \mathbf{\hat A}_\emptyset) \quad \text{and} \quad \mathbf{\hat X}_\bi = (\hat N_\bi+1, \mathbf{\hat A}_\bi), \quad \bi \neq \emptyset,$$
and let $\hat T(\mathbf{\hat A}) = \{ \mathbf{\hat X}_\bi: \bi \in \hat T \}$. We use 

For a CM, the distribution of the full marks is given by:
\begin{align*}
\mathbb{P}_n\left( \mathbf{\hat X}_\emptyset \in \cdot \right) &= \frac{1}{n} \sum_{i=1}^n 1((D_i, {\bf a}_i) \in \cdot),  \\
\mathbb{P}_n\left( \mathbf{\hat X}_\bi \in \cdot \right) &=  \sum_{i=1}^n \frac{D_i}{L_n}1((D_i, {\bf a}_i) \in \cdot), \quad \bi\neq \emptyset .
\end{align*}
For the IR model, first let $\{b_n\}$ be a sequence such that $b_n \xrightarrow{P} \infty$ and $b_n/\sqrt{n} \xrightarrow{P} 0$ as $n\to\infty$, and use it to define $\bar W_i = W_i \wedge b_n$ and 
$$\Lambda_n = \sum_{i=1}^n \bar W_i.$$
The marks on the coupled marked Galton-Watson process are given by:
\begin{align*}
\mathbb{P}_n\left( \mathbf{\hat X}_\emptyset \in \cdot \right) &= \frac{1}{n} \sum_{i=1}^n P( (D_i, {\bf a}_i) \in \cdot | {\bf a}_i),  \\
\mathbb{P}_n\left( \mathbf{\hat X}_\bi \in \cdot \right) &=  \sum_{i=1}^n \frac{\bar W_i}{\Lambda_n} P( (D_i+1, {\bf a}_i) \in \cdot | {\bf a}_i ) , \quad \bi \neq \emptyset,
\end{align*}
where conditionally on ${\bf a}_i$, $D_i$ is a Poisson r.v.~with mean $\Lambda_n \bar W_i/(\theta n)$. 

We will also need to extend our definition of an isomorphism for marked graphs.

\begin{defn} \label{D.MarkIsomorphic}
A graph $G(V,E)$ is called a {\em vertex-weighted} graph if each of its vertices has a mark (weight) assigned to it. 
We say that the two vertex-weighted simple graphs $G(V, E)$ and $G(V', E')$ are {\em isomorphic} if there exists a bijection $\sigma: V \to V'$ such that edge $(i,j) \in E$ if and only if edge $(\sigma(i), \sigma(j)) \in E'$, and in addition, the marks of $i$ and $\sigma(i)$ are the same. 
For vertex-weighted multigraphs, we say that $G(V,E)$ and $G'(V', E')$ are {\em isomorphic} if there exists a bijection $\sigma: V \to V'$ such that $l(i) = l(\sigma(i))$ and $e(i,j) = e(\sigma(i), \sigma(j))$ for all $i \in V$ and all $(i,j) \in E$, where $l(i)$ is the number of self-loops of vertex $i$ and $e(i,j)$ is the number of edges from vertex $i$ to vertex $j$, and in addition, the marks of $i$ and $\sigma(i)$ are the same. In both cases, we write $G \simeq G'$. 
\end{defn}

The intermediate coupling theorem is given below.

\begin{theo} \label{T.Coupling}
Suppose $G(V_n, E_n)$ is either a CM or an IR satisfying Assumption~\ref{A.PrimitivesU}. Then, for $\mathcal{G}^{(k)}_I({\bf a})$ the depth-$k$ neighborhood of a uniformly chosen vertex $I \in V_n$, there exists a marked Galton-Watson tree $\hat T^{(k)} (\mathbf{\hat A})$ restricted to its first $k$ generations, whose root corresponds to vertex $I$, and 
such that for any fixed $k \geq 1$,
$$\mathbb{P}_n\left( \mathcal{G}_I^{(k)}(\mathbf{a}) \not\simeq \hat T^{(k)}(\mathbf{\hat A}) \right) \xrightarrow{P} 0, \qquad n \to \infty.$$
\end{theo}

The proof of Theorem~\ref{T.Coupling} is given separately for the two models being considered, the CM and the IR. 

\subsubsection{Coupling for the configuration model} \label{S.CouplingCM}

To explore the neighborhood of depth $k$ of vertex $I \in G(V_n, E_n)$ we start by labeling the set of $L_n$ stubs in such a way that stubs $\{1, \dots, D_1\}$ belong to vertex 1, stubs $\{ D_1+1, \dots, D_1 + D_2\}$ belong to vertex 2, and in general, stubs $\{ D_1 + \dots + D_{m-1} + 1, \dots, D_1 + D_m\}$ belong to vertex $m$. 

For any $k \geq 0$ define the sets:
\begin{align*}
A_k &= \text{set of vertices in $G(V_n, E_n)$ at distance $k$ from vertex $I$.} \\
J_k &= \text{set of stubs belonging to vertices in $A_k$.} \\ 
V_k &= \bigcup_{r=0}^k A_r. \\
\hat A_k &= \text{set of nodes in $\hat T$ at distance $k$ from the root $\emptyset$.} \\
\hat V_k &= \bigcup_{r=0}^k \hat A_r.
\end{align*}
These sets will be constructed as we explore the graph.

To do a breadth-first exploration of $G(V_n, E_n)$ we start by selecting vertex $I$ uniformly at random. Next, let $J_0$ denote the set of stubs belonging to vertex $I$ and set $A_0 = \{ I \}$. For $k \geq 1$, Step $k$ in the exploration will identify all the stubs belonging to nodes in $A_k$. 

{\em Step $k$, $k \geq 1$:} 
\begin{itemize}
\item[a.] Initialize the sets $A_k = J_k = \varnothing$.
\item[b.] For each vertex $i \in A_{k-1}$:
\begin{itemize}
\item[i.] For each of the unpaired stubs of vertex $i$:
\begin{itemize}
\item [1)] Pick an unpaired stub of vertex $i$ and sample uniformly at random a stub from the $L_n$ available. If the chosen stub is the stub currently being paired or if it had already been paired, sample again until an unpaired stub is sampled. 
\item[2)] If the chosen stub belongs to vertex $j$, draw an edge between vertices $i$ and $j$ using the chosen stub. If vertex $j$ had not yet been discovered, add it to $A_k$ and add all of its unpaired stubs to $J_k$. 
\end{itemize}
\end{itemize}
\end{itemize}
The exploration terminates at Step $k$ if $J_k = \varnothing$, at which point the component of $I$ will have been fully explored.

To couple the construction of $\hat T$ initialize $\hat A_0 = \{ \emptyset\}$ and identify $\emptyset$ with vertex $I$ in $G(V_n, E_n)$ and set $\hat N_\emptyset = D_I$, ${\bf a}_\emptyset = {\bf a}_I$. For $k \geq 1$, Step $k$ in the construction will identify all the nodes in $A_k$ by adding nodes in agreement with the exploration of the graph. Each node that is added to the tree will have a number of stubs equal to the total number of stubs of the corresponding vertex, minus one (the one being used to create the edge), regardless of whether some of those stubs may already have been paired. 

 {\em Step $k$, $k \geq 1$:} 
\begin{itemize}
\item[a.] Initialize the set $\hat A_k = \emptyset$. 
\item[b.] For each node ${\bf i} = (i_1, \dots, i_{k-1}) \in \hat A_{k-1}$:
\begin{itemize}
\item[i.] For each $1 \leq r \leq \hat N_{\bf i}$:
\begin{itemize}
\item[1)] Pick a stub uniformly at random from the $L_n$ available. 
\item[2)] If the chosen stub belongs to vertex $j$, then add node $({\bf i},r)$ to $\hat A_k$ and set $\hat N_{({\bf i}, r)} = D_j - 1$, $\mathbf{\hat A}_{({\bf i},r)} = {\bf a}_j$. 
\end{itemize}
\end{itemize}
\end{itemize}
This process will end in Step $k$ if $\hat N_{\bf i} = 0$ for all ${\bf i} \in \hat A_k$, or it may continue indefinitely. 

Note that the coupling relies on using the same uniform random numbers in step (b)(i)(1) for the two constructions. Specifically, there is one uniform for each of the $L_n$ stubs which is used in step $(b)(i)(1)$ of the tree construction, and in the first pick in the acceptance-rejection step $(b)(i)(1)$ of the graph exploration; in case of a rejection, independent uniforms are used until a stub is accepted.

\begin{defn}
We say that the coupling breaks in generation $\tau = k$ if:
\begin{itemize}
\item The first time we have to resample a stub in step $(b)(i)(1)$ occurs while exploring a stub belonging to a vertex in $A_{k-1}$; or
\item If given that the above has not happened, a stub belonging to a vertex in $A_{k-1}$ is paired with a stub belonging to a previously encountered vertex (this vertex could be in either $A_{k-1}$ or the current set $A_k$). 
\end{itemize}
\end{defn}

{\bf Note:} The exploration of the component of depth $k$ of vertex $I$ in $G(V_n, E_n)$ and the construction of the first $k$ generations of the tree $\hat T$ will be identical provided $\tau > k$, meaning $\mathcal{G}_I^{(k)}(\mathbf{a}) \simeq \hat T^{(k)}(\mathbf{\hat A})$  in the sense of Definition~\ref{D.MarkIsomorphic}. This particular construction is standard in the analysis of the configuration model, with small variations depending on whether one needs to keep track of completed generations or simply the number of stubs that have been paired. We refer the reader to Chapter 4 in \cite{ Hofstad2} for the stub by stub version and a full history of the model. Earlier versions of the coupling imposed finite second moment conditions that more recent proofs can omit (see, e.g., \cite{Hofstad2005}).

\begin{proof}[Proof of Theorem~\ref{T.Coupling} ($m=1$) for the CM]
From the observation made above, it suffices to show that the exploration of the $k$-neighborhood of vertex $I$ does not require us to resample any stub in step $(b)(i)(1)$ nor samples a stub belonging to a vertex that had already been discovered. To compute the probability of successfully completing $k$ generations in $\hat T$ before the coupling breaks, write:
\begin{align*}
\mathbb{P}_n \left( \mathcal{G}_I^{(k)}(\mathbf{a}) \neq \hat T^{(k)}(\mathbf{\hat A}) \right) &\leq \mathbb{P}_n(\tau \leq k) .
\end{align*}

The coupling breaks the first time we draw a stub belonging to a vertex that has already been explored: either a stub already paired, or one that is unpaired but already attached to the graph. The number of paired stubs when exploring a vertex in $A_{r-1}$ is smaller or equal than $2 \sum_{j=1}^{r} |A_j| +  |J_{r}| $, which corresponds to two stubs each for the vertices at distance at most $r$ of $I$ and the unpaired stubs belonging to nodes in $J_{r}$. Note that up to the moment that the coupling breaks, we have $|A_j| = |\hat A_j|$ for all $0 \leq j \leq r$, and $|J_r| = |\hat A_{r+1}|$, so the probability that we break the coupling while exploring a vertex in $A_{r-1}$ is smaller or equal than
$$P_{r} := \frac{2}{L_n} \sum_{j=1}^{r+1} |\hat A_j| \leq \frac{2 |\hat V_{r+1}|}{L_n}, \qquad r \geq 1.$$ 

It follows that for any $a_n > 0$, 
\begin{align*}
\mathbb{P}_n(\tau \leq k) &= \mathbb{P}_n(\tau \leq k, |\hat V_{k+1}| \leq a_n ) + \mathbb{P}_n ( |\hat V_{k+1}| > a_n) \\
&\leq \sum_{r=1}^k \mathbb{P}_n(\tau = r, |\hat V_{r+1}| \leq a_n ) + \mathbb{P}_n ( |\hat V_{k+1}| > a_n) \\
&\leq \sum_{r=1}^k \mathbb{P}_n \left(\text{Bin}(\hat A_{r-1}, P_r) \geq 1, |\hat V_{r+1}| \leq a_n \right) + \mathbb{P}_n ( |\hat V_{k+1}| > a_n) \\
&\leq \sum_{r=1}^k \mathbb{P}_n \left(\text{Bin}(a_n, 2a_n/L_n) \geq 1 \right) + \mathbb{P}_n ( |\hat V_{k+1}| > a_n) \\
&\leq \sum_{r=1}^k \frac{2a_n^2}{L_n} + \mathbb{P}_n ( |\hat V_{k+1}| > a_n),
\end{align*}
where $\text{Bin}(n,p)$ represents a binomial random variable with parameters $(n,p)$. Hence, we have
$$\mathbb{P}_n \left( \mathcal{G}_I^{(k)}(\mathbf{a})  \not\simeq \hat T^{(k)}(\mathbf{\hat A}) \right)  \leq \mathbb{P}_n(\tau \leq k) \leq  \frac{2k a_n^2}{L_n} + \mathbb{P}_n ( |\hat V_{k+1}| > a_n).$$

To analyze the last probability we use the first part of Theorem~\ref{T.TreeToTree} to obtain that for any fixed $k \geq 1$ there exists a tree $\mathcal{T}^{(k)}$ of depth $k$, whose distribution does not depend on $\mathscr{F}_n$, such that
$$\mathbb{P}_n \left( \hat T^{(k)} \not\simeq \mathcal{T}^{(k)} \right) \xrightarrow{P} 0,$$
as $n \to \infty$. Let $|\mathcal{A}_k|$ denote the size of the $k$th generation of that tree, define $|\mathcal{V}_{k+1}| = \sum_{j=0}^{k+1} |\mathcal{A}_j|$, and note that
\begin{align*}
\mathbb{P}_n \left( \mathcal{G}_I^{(k)}(\mathbf{a})  \not\simeq \hat T^{(k)}(\mathbf{\hat A}) \right)  &\leq   \frac{2k a_n^2}{L_n} + P ( |\mathcal{V}_{k+1}| > a_n) + \mathbb{P}_n \left( \hat T^{(k)} \not\simeq \mathcal{T}^{(k)} \right).
\end{align*}
Choosing $a_n \xrightarrow{P} \infty$ so that $a_n^2/n \xrightarrow{P} 0$ as $n \to \infty$, and observing that $|\mathcal{V}_{k+1}| < \infty$ a.s.,~completes the proof. 
\end{proof}

\subsubsection{Coupling for the inhomogeneous random graph} \label{S.CouplingIR}

We will couple the exploration of the component of vertex $I \in G(V_n, E_n)$ with a marked multi-type Galton-Watson process with $n$ types, one for each vertex in $G(V_n, E_n)$. A node of type $i \in \{1, 2, \dots, n\}$ in the tree will have a Poisson number of offspring of type $j$ with mean:
$$q_{ij}^{(n)} = \frac{\bar W_i \bar W_j}{\theta n}, \qquad 1\leq j \leq n.$$

Similarly as in the case of the CM, define:
\begin{align*}
A_k &= \text{set of vertices in $G(V_n, E_n)$ at distance $k$ from vertex $I$.} \\
V_k &= \bigcup_{r=0}^k A_r. \\
\hat A_k &= \text{set of nodes in $\hat T$ at distance $k$ from the root $\emptyset$.} \\
\hat B_k &= \text{set of types of nodes in $\hat A_k$.} \\
\hat V_k &= \bigcup_{r=0}^k \hat A_r.
\end{align*}

We will again do a breadth-first exploration of $G(V_n, E_n)$ starting from a uniformly chosen vertex $I$. To start, let $\{ U_{ij}: i,j \geq 1\}$ be a sequence of i.i.d.~Uniform$[0,1]$ random variables, independent of $\mathscr{F}_n$. We will use this sequence of i.i.d.~uniforms to realize the Bernoulli random variables that determine the presence/absence of edges in $G(V_n, E_n)$. Set $A_0 = \{ I \}$ and initialize the set $J = \varnothing$; the set $J$ will keep track of the vertices that have been fully explored (all its potential edges realized), and will coincide with $V_{k-1}$ at the end of Step $k$. 

{\em Step $k$, $k\geq 1$:} 
\begin{itemize}
\item[a.] Initialize the set $A_k = \varnothing$.
\item[b.] For each vertex $i \in A_{k-1}$:
\begin{itemize}
\item[i.] Let $X_{ij} = 1(U_{ij} > 1- p_{ij}^{(n)})$ for each $j \in \{1, 2, \dots, n\} \setminus J$.
\item[ii.] If $X_{ij} = 1$ draw an edge between vertices $i$ and $j$ and add vertex $j$ to $A_k$. 
\item[iii.] Add vertex $i$ to set $J$. 
\end{itemize}
\end{itemize}
The exploration terminates at the end of Step $k$ if $A_k = \varnothing$, at which point the component of $I$ will have been fully explored.

To couple the construction of $\hat T$ initialize $\hat A_0 = \{ \emptyset\}$ and identify $\emptyset$ with vertex $I$ in $G(V_n, E_n)$ as before; let $\hat B_0 = \{ I \}$. To construct the tree, we will sample for a node of type $i$ a Poisson number of offspring of type $j$ for each $j \in \{1, \dots, n\}$. To do this, let $G(\cdot; \lambda)$ be the cumulative distribution function of a Poisson random variable with mean $\lambda$, and let $G^{-1}(u; \lambda) = \inf\{ x \in \mathbb{R}: G(x; \lambda) \geq u\}$ denote its pseudoinverse. In order to keep the tree coupled with the exploration of the graph we will use the same sequence of i.i.d.~uniform random variables used to sample the edges in the graph. Initialize the set $\hat J = \varnothing$, which will keep track of the types that have appeared and whose offspring have been sampled. The precise construction is given below:

 {\em Step $k$, $k \geq 1$:} 
\begin{itemize}
\item[a.] Initialize the sets $\hat A_k =  \hat B_k = \varnothing$. 
\item[b.] For each node ${\bf i} = (i_1, \dots, i_{k-1}) \in \hat A_{k-1}$:
\begin{itemize}
\item[i.] If ${\bf i}$ has type $t \notin \hat J$:
\begin{itemize}
\item[1)] For each type $j \in \{1, \dots, n\} \setminus \hat J$ let $Z_{tj} = G^{-1}(U_{tj}; q_{tj}^{(n)})$, and create $Z_{tj}$ children of type $j$ for node ${\bf i}$. If $Z_{tj} \geq 1$, create $Z_{tj}$ children of type $j$ for node ${\bf i}$, each with node attribute equal to $\mathbf{a}_j$, and add $j$ to set $\hat B_k$. 
\item[2)] For each type $j \in \hat J$ sample $Z_{tj}^* \sim$ Poisson$(q_{tj}^{(n)})$, independently of the sequence $\{U_{ij}: i,j \geq 1\}$ and any other random variables. If $Z_{tj}^* \geq 1$ create $Z_{tj}^*$ children of type $j$ for node ${\bf i}$, each with attribute equal to $\mathbf{a}_j$.
\item[3)] Randomly shuffle all the children created in steps (b)(i)(1) and (b)(i)(2) and give them labels of the form $({\bf i}, j)$, then add the labeled nodes to set $\hat A_k$. The node attributes will be denoted $\mathbf{\hat A}_{(\bi,j)} = \mathbf{a}_j$. (The shuffling avoids the label from providing information about its type). 
\item[4)] Add type $t$ to set $\hat J$. 
\end{itemize}
\item[ii.] If ${\bf i}$ has type $t\in \hat J$:
\begin{itemize}
\item[1)] For each type $j \in \{1, \dots, n\}$ sample $Z_{tj}^* \sim$ Poisson$(q_{tj}^{(n)})$, independently of the sequence $\{U_{ij}: i,j \geq 1\}$ and any other random variables; create $Z_{tj}^*$ children of type $j$ for node ${\bf i}$, each with attribute equal to $\mathbf{a}_j$.
\item[2)] Randomly shuffle all the children created in step (b)(ii)(1) and give them labels of the form $({\bf i},j)$, attributes $\mathbf{\hat A}_{(\bi,j)} = \mathbf{a}_j$, and add the labeled nodes to set $\hat A_k$. 
\end{itemize}
\end{itemize}
\end{itemize}
This construction may continue indefinitely, or may terminate at the end of Step $k$ if $\hat A_k = \varnothing$. 

We point out that this coupling is not standard, since in most of the literature on IR models the coupling is done between the binomial distribution (the degree of a vertex) and its coupled Poisson limit (see \cite{Boll_Jan_Rio_07} or Chapter 3 in \cite{Hofstad2}, where the proof uses the moments method). Moreover, most of the existing couplings consider first a multi-type Galton Watson process with finitely many types, and then use a monotonicity argument to obtain the general case. The coupling described above avoids the need for this second step since the number of types grows as $n \to \infty$, and since it is based on coupling the individual Bernoulli random variables (edges) with their Poisson counterparts, it  allows us to keep track of the vertex marks at no additional effort.

\begin{defn}
We say that the coupling breaks in generation $\tau = k$ if for any node in $\hat A_{k-1}$ either:
\begin{itemize}
\item In step (b)(i)(1) we have $Z_{tj} \neq X_{tj}$ for some $j \in \{1, \dots, n\} \setminus  \hat J $;
\item In step (b)(i)(1) we have $Z_{tj} \geq 1$ for some $j \in  (\hat B_{k-1} \cup \hat B_k) \setminus \hat J$, in which case a cycle or self-loop is created; or,
\item In step (b)(i)(2) we have $Z_{tj}^* \geq 1$ for some $j \in \hat J$. 
\end{itemize}
\end{defn}

We start by proving the following preliminary result. Throughout this section, let 
$$\Delta_n := \int_0^1 \left| F_n^{-1}(u) - F^{-1}(u) \right| du \leq W_1(\nu_n, \nu),$$
where $F_n(x) = \frac{1}{n} \sum_{j=1}^n 1(W_i \leq x)$ and $F(x) = P(W \leq x)$. We also use the notation $X_n = O_P( x_n)$ as $n \to \infty$ to mean that there exists a random variable $Y_n$ such that $|X_n |\leq_\text{s.t.} Y_n$ and $Y_n/x_n \xrightarrow{P} K$ for some finite constant $K$. 

\begin{lemma} \label{L.EdgeDiscrepancy}
For any $1 \leq i \leq n$ we have
\begin{align*}
 \mathbb{P}_n\left( \max_{1 \leq j \leq n, j \neq i} |X_{ji} - Z_{ji}| \geq 1 \right) &\leq \min\left\{ 1, 1(W_i > b_n) + \mathcal{P}_n(i) + \bar W_i \eta_n \right\}, 
\end{align*}
where 
$$\mathcal{P}_n(i) =  \sum_{1 \leq j \leq n, j \neq i} \left| p_{ji}^{(n)} - (r_{ji}^{(n)} \wedge 1) \right| , $$
$$\eta_n = ( \Delta_n + g(b_n) +  b_n^2/n + b_n^2 \Delta_n/(\theta n))/\theta, \qquad  \text{and} \qquad g(x) = E[(W - x)^+].$$ 
\end{lemma}

\begin{proof}
Let $R_{ij} = 1(U_{ij} > 1- r_{ij}^{(n)})$ with $r_{ij}^{(n)} = W_i W_j/(\theta n)$. The union bound gives:
\begin{align*}
\mathbb{P}_n\left( \max_{1 \leq j \leq n, j \neq i} |X_{ij} - Z_{ij}| \geq 1 \right) &\leq  1(W_i > b_n) + 1(W_i \leq b_n) \sum_{1 \leq j \leq n, j \neq i} \mathbb{P}_n( |X_{ij} - Z_{ij}| \geq 1).
\end{align*}
Now note that
\begin{align*}
\mathbb{P}_n( |X_{ij} - Z_{ij}| \geq 1) &= \mathbb{P}_n( |X_{ij} - Z_{ij}| \geq 1, |X_{ij} - R_{ij}| \geq 1) \\
&\hspace{5mm} + \mathbb{P}_n( |X_{ij} - Z_{ij}| \geq 1, |X_{ij} - Ri_{ij}| = 0) \\
&\leq  \mathbb{P}_n( |X_{ij} - R_{ij}| \geq 1) + \mathbb{P}_n( |R_{ij} - Z_{ij}| \geq 1). 
\end{align*}
The first probability can be computed to be:
\begin{align*}
\mathbb{P}_n( |X_{ij} - Z_{ij}| \geq 1) &= | p_{ij}^{(n)} - (r_{ij}^{(n)} \wedge 1)|. 
\end{align*}

To analyze each of probabilities involving $R_{ij}$ and $Z_{ij}$, note that 
\begin{align*}
\mathbb{P}_{n} \left( |R_{ij} - Z_{ij}| \geq 1 \right) &= \mathbb{P}_n(R_{ij} = 0, Z_{ij} \geq 1) + \mathbb{P}_n(R_{ij} = 1, Z_{ij} = 0) + \mathbb{P}_n(R_{ij} = 1, Z_{ij} \geq 2) \notag \\
&= \left( 1- (1 \wedge r_{ij}^{(n)}) - e^{-q_{ij}^{(n)}} \right)^+ +  \left( e^{-q_{ij}^{(n)}}  - 1 + (1 \wedge r_{ij}^{(n)})  \right)^+ \notag \\
&\hspace{5mm} + \min\left\{ 1 - e^{-q_{ij}^{(n)}} (1 + q_{ij}^{(n)}) , \, (1 \wedge r_{ij}^{(n)} )  \right\}  \notag\\
&= \left| 1- (1 \wedge r_{ij}^{(n)})  - e^{-q_{ij}^{(n)}} \right| + \min\left\{ (1 \wedge r_{ij}^{(n)}), \, e^{-q_{ij}^{(n)}} ( e^{q_{ij}^{(n)}} - 1 - q_{ij}^{(n)} ) \right\} .
\end{align*}
Now use the inequalities $e^{-x} \geq 1-x$, $e^{-x} - 1 + x \leq x^2/2$ and $e^x - 1 - x \leq x^2 e^x/2$ for $x \geq 0$, to obtain that
\begin{align*}
\mathbb{P}_{n} \left( |X_{ij} - Z_{ij}| \geq 1 \right) &\leq   r_{ij}^{(n)} - q_{ij}^{(n)} + \left| 1 - q_{ij}^{(n)} - e^{-q_{ij}^{(n)}} \right| + e^{-q_{ij}^{(n)}} ( e^{q_{ij}^{(n)}} - 1 - q_{ij}^{(n)} ) \\
&= r_{ij}^{(n)} - q_{ij}^{(n)} + e^{-q_{ij}^{(n)}} - 1 + q_{ij}^{(n)} +  e^{-q_{ij}^{(n)}} ( e^{q_{ij}^{(n)}} - 1 - q_{ij}^{(n)} ) \\
&\leq  r_{ij}^{(n)} - q_{ij}^{(n)} + (q_{ij}^{(n)})^2.
\end{align*}

It follows that
\begin{align*}
&1(W_i \leq b_n) \sum_{1 \leq j \leq n, j \neq i} \mathbb{P}_n( |X_{ij} - Z_{ij}| \geq 1) \\
&\leq 1(W_i \leq b_n)  \sum_{1 \leq j \leq n, j \neq i} \left( | p_{ij}^{(n)} - (r_{ij}^{(n)} \wedge 1)| + r_{ij}^{(n)} - q_{ij}^{(n)} + (q_{ij}^{(n)})^2   \right)  \\
&\leq \mathcal{P}_n(i) + \sum_{1 \leq j \leq n, j \neq i} \frac{\bar W_i (W_j - {\bar W}_j)}{\theta n}  + \frac{({\bar W}_i)^2}{(\theta n)^2} \sum_{1 \leq j \leq n, j \neq i} ({\bar W}_j)^2 \\
&\leq \mathcal{P}_n(i) + \frac{\bar W_i}{\theta n} \sum_{j=1}^n (W_j - b_n)^+  + \frac{({\bar W}_i)^2 b_n \Lambda_n}{(\theta n)^2}.
\end{align*}
To further bound the second term note that if we let $W^{(n)}$ denote a random variable distributed according to $F_{n}$ and $W$ a random variable distributed according to $F$, then
$$\frac{1}{n} \sum_{j=1}^n (W_j - b_n)^+ = \mathbb{E}_n\left[ (W^{(n)} - b_n)^+ \right] \leq  \mathbb{E}_n \left[ |W^{(n)} -W|+ (W - b_n)^+ \right] = \Delta_n + g(b_n).$$
And for the last term, 
$$\frac{({\bar W}_i)^2 b_n \Lambda_n}{(\theta n)^2} \leq \frac{\bar W_i b_n^2}{\theta^2 n} \cdot \mathbb{E}_n\left[ W^{(n)}  \right]  \leq \frac{\bar W_i b_n^2}{\theta^2 n} \left( \Delta_n + E[W] \right) .$$

We conclude that for $\mathcal{E}_n$ as defined in the statement of the lemma, 
\begin{align*}
1(W_i \leq b_n) \sum_{1 \leq j \leq n, j \neq i} \mathbb{P}_n( |X_{ij} - Z_{ij}| \geq 1) &\leq  \mathcal{P}_n(i) + \frac{\bar W_i}{\theta} ( \Delta_n + g(b_n)) + \frac{\bar W_i  b_n^2}{\theta^2 n} (\Delta_n + E[W^-]) \\
&\leq \mathcal{P}_n(i) + \eta_n \bar W_i,
\end{align*}
which in turn yields
\begin{align*}
\mathbb{P}_n\left( \max_{1 \leq j \leq n, j \neq i} |X_{ij} - Z_{ij}| \geq 1 \right) &\leq \min\left\{ 1, 1(W_i > b_n) + \mathcal{P}_n(i) + \eta_n \bar W_i \right\}  .
\end{align*}
\end{proof}

\begin{proof}[Proof of Theorem~\ref{T.Coupling} ($m=1$) for the IR]
We start by defining the following events:
\begin{align*}
F_{i}(I, J, L) &=  \left\{  \max_{j \in I} |X_{ji} - Z_{ji}|= 0, \, \sum_{j \in J} Z_{ji}^* + \sum_{ j \in L}  Z_{ji}  = 0  \right\}, \\
\mathbb{B}_{i} &= \{ \text{ current set $\hat B_{k-1} \cup \hat B_k$ when the neighbors of $i \in A_{k-1}$ are explored} \} ,\\
\mathbb{J}_i &= \{ \text{ current set $J$ when the neighbors of $i$ are explored} \} ,\\
H_k &= \bigcap_{i \in A_{k-1}} F_i(\{ 1, \dots, n\} \setminus \mathbb{J}_i, \mathbb{J}_i, \mathbb{B}_{i} \setminus \mathbb{J}_i  ), \\
M_k &= \{ |\hat V_k| \leq s_n\}. 
\end{align*}
Next, note that
\begin{align*}
\mathbb{P}_{n} (\tau \leq k) &\leq \mathbb{P}_{n} (\tau \leq k, M_k ) + \mathbb{P}_{n} (M_k^c) \\
&\leq \sum_{r=1}^k  \frac{1}{n} \sum_{i=1}^n \mathbb{P}_{n,i} (\tau =r, M_r ) + \mathbb{P}_{n} (M_k^c) ,
\end{align*}
where the last probability can be bounded using the first part of Theorem~\ref{T.TreeToTree} as it was done at the end of the proof of Theorem~\ref{T.Coupling} for the CM. Specifically,
$$\mathbb{P}_{n} (M_k^c) \leq P\left( |\mathcal{V}_{k+1}| > s_n \right) + \mathbb{P}_n\left( \hat T^{(k)} \not\simeq \mathcal{T}^{(k)} \right),$$
where $|\mathcal{V}_{k+1}| = \sum_{j=0}^{k+1} |\mathcal{A}_j| < \infty$ a.s.~and the distribution of $\mathcal{T}^{(k)}$ does not depend on $\mathscr{F}_n$. 

Now note that for any $r \geq 1$, 
\begin{align*}
\mathbb{P}_{n,i}\left(\tau = r, M_{r} \right) &= \mathbb{P}_{n,i}\left( M_{r} \cap  \bigcap_{m=1}^{r-1}  H_{m} \cap H_r^c \right) , \\
\end{align*}
with the convention that $\bigcap_{m=1}^0 H_m = \Omega$.  Let $\mathcal{F}_t$ denote the sigma-algebra that contains the history of the exploration process in the graph as well as that of its coupled tree, up to the end of Step $t$ of the graph exploration process. It follows that we can write:
\begin{align*}
\mathbb{P}_{n,i}(\tau = r, M_{r}) &= \mathbb{E}_{n,i}\left[ 1\left( M_{r-1}  \cap  \bigcap_{m=1}^{r-1} H_{m}  \right) \mathbb{P}_n( M_r \cap H_r^c | \mathcal{F}_{r-1}) \right] .
\end{align*}
To analyze the conditional probability inside the expectation above note that conditionally on $\mathcal{F}_{r-1}$, the set $A_{r-1}$ is known, and recall that the set $J = V_{r-2}$ at the beginning of Step $r$ (assuming $r \geq 2$, otherwise, $J = \varnothing$).  Therefore, by the union bound and the independence among the edges, we have:
\begin{align*}
\mathbb{P}_n( M_r \cap H_r^c | \mathcal{F}_{r-1}) &= \mathbb{P}_n\left( \left. M_r \cap \bigcup_{i \in A_{r-1}} F_i(\{ 1, \dots, n\} \setminus \mathbb{J}_i, \mathbb{J}_i, \mathbb{B}_{i} \setminus \mathbb{J}_i  )^c \right| \mathcal{F}_{r-1} \right)  \\
&\leq \sum_{i \in A_{r-1} }  \mathbb{P}_n\left( \left. M_r \cap F_i(\{ 1, \dots, n\} \setminus \mathbb{J}_i, \mathbb{J}_i, \mathbb{B}_{i} \setminus \mathbb{J}_i  )^c \right| \mathcal{F}_{r-1} \right)  \\
&\leq \sum_{i \in A_{r-1} }   \min\left\{  1, \, \mathbb{P}_n\left( \left.  \max_{j \in \{ 1, \dots, n\} \setminus \mathbb{J}_i} |X_{ji} - Z_{ji}| \geq 1   \right| \mathcal{F}_{r-1} \right) \right. \\
&\hspace{5mm} + \left.  \mathbb{P}_n\left( \left. M_r \cap \left\{ \sum_{j \in \mathbb{J}_i} Z_{ji}^* +  \sum_{ j \in \mathbb{B}_i \setminus \mathbb{J}_i}  Z_{ji} \geq 1\right\}   \right| \mathcal{F}_{r-1} \right)  \right\}   .
\end{align*}
Now use the independence of the edges from the rest of the exploration process and Lemma~\ref{L.EdgeDiscrepancy} to obtain that
\begin{align*}
 \mathbb{P}_n\left(  \left. \max_{\{1, \dots, n\} \setminus \mathbb{J}_i i} |X_{ji} - Z_{ji}| \geq 1 \right| \mathcal{F}_{r-1} \right) &\leq  \mathbb{P}_n\left(  \max_{1 \leq j \leq n, j \neq i} |X_{ji} - Z_{ji}| \geq 1  \right) \\
 &\leq 1(W_i > b_n) + \mathcal{P}_n(i) + \bar W_i \eta_n .
\end{align*}
Next, condition further on the exploration up to the moment we are about to explore the  neighbors of $i$, and use the independence of the edges from the rest of the exploration process to obtain that
\begin{align*}
& \mathbb{P}_n\left(  \left. M_r \cap \left\{ \sum_{j \in \mathbb{J}_i} Z_{ji}^* +  \sum_{ j \in \mathbb{B}_i \setminus \mathbb{J}_i}  Z_{ji} \geq 1\right\}   \right| \mathcal{F}_{r-1}\right) \\
&\leq \mathbb{E}_n\left[ \left. 1( |\hat V_r| \leq s_n) \left(  1 - e^{ -  \sum_{j \in \mathbb{B}_i} q_{ji}^{(n)} }  \right) \right| \mathcal{F}_{r-1} \right] \\
 &= \mathbb{E}_n\left[ \left. 1(|\hat V_r| \leq s_n) \left( 1 - e^{-\frac{{\bar W}_i}{\theta n}   \sum_{j \in \cup \mathbb{B}_i} {\bar W}_j  } \right) \right| \mathcal{F}_{r-1} \right]  \\
 &\leq \mathbb{E}_n\left[ \left. 1(|\hat V_r| \leq s_n) \left(  1 - e^{-\frac{b_n {\bar W}_i}{\theta n}  |\mathbb{B}_i| }  \right) \right| \mathcal{F}_{r-1}  \right] \\
 &\leq  \frac{b_n {\bar W}_i}{\theta n}  s_n ,
\end{align*}
where in the last inequality we used $1 - e^{-x} \leq x$ for $x \geq 0$ and $|\mathbb{B}_i| \leq |\hat V_r| \leq s_n$. 

It follows that 
\begin{align*}
\mathbb{P}_{n,i}(\tau = r, M_{r})  &\leq  \mathbb{E}_{n,i}\left[ 1\left( M_{r-1}  \cap  \bigcap_{m=1}^{r-1} H_{m} \right)  \sum_{j \in A_{r-1}}  \min\left\{1, \, 1(W_i > b_n) +  \mathcal{P}_n(j) + \bar W_i \eta_n \phantom{\frac{\bar W^-}{\theta}} \right. \right.  \\
&\hspace{5mm} \left.  \left. +  \frac{b_n {\bar W}_i}{\theta n}  s_n  \right\}  \right] .
\end{align*}
To analyze this remaining expectation we note that on the event $ \bigcap_{m=1}^{r-1}  H_{m} $  the coupling has not broken yet, and therefore we can can replace $A_{r-1}$ with its tree counterpart $\hat A_{r-1}$.  Also, note that by Lemma~3.4 in \cite{Lee_Olv_20} we have that the types of the nodes in each of the sets $\hat A_k$ are independent of the type of their parents. We will then identify the nodes in $\hat A_{r-1}$ as $\{Y_1, \dots, Y_{|\hat A_{r-1}|}\}$, where for any $t \geq 1$, 
$$\mathbb{P}_n(Y_{t} = j) = \frac{{\bar W}_j}{\Lambda_n}, \quad j = 1, 2, \dots, n.$$

It follows that
\begin{align*}
\mathbb{P}_{n,i}(\tau = r, M_{r})  &\leq  \mathbb{E}_{n,i}\left[ 1\left( M_{r-1} \right) \sum_{t=1}^{|\hat A_{r-1}|}  \min\left\{ 1,  \,  1(W_{Y_t} > b_n) + \mathcal{P}_n(Y_t ) +  \bar W_{Y_t} \eta_n  + \frac{b_n {\bar W}_{Y_t}}{\theta n} s_n   \right\}   \right] \\
&\leq \mathbb{E}_{n,i}\left[  \sum_{t=1}^{\lfloor s_n \rfloor}  \min\left\{ 1,  \,  1(W_{Y_t} > b_n) + \mathcal{P}_n(Y_t ) +  \bar W_{Y_t} \eta_n  + \frac{b_n {\bar W}_{Y_t}}{\theta n} s_n   \right\}   \right] \\
&\leq \sum_{t=1}^{\lfloor s_n \rfloor} \mathbb{E}_{n,i}\left[   1(W_{Y_t} > b_n) + \mathcal{P}_n(Y_t ) +  \bar W_{Y_t} \eta_n  + \frac{b_n {\bar W}_{Y_t}}{\theta n} s_n   \right] \\
&= \lfloor s_n \rfloor \mathbb{E}_n \left[   1(W_{Y_1} > b_n) + \mathcal{P}_n(Y_1 ) +  \bar W_{Y_1} \eta_n  + \frac{b_n {\bar W}_{Y_1}}{\theta n} s_n   \right] .
\end{align*}

To compute the last expectation, let $(W^{(n)}, W)$ be constructed according to an optimal coupling of $F_n$ and $F$. Let $\bar W^{(n)} = W^{(n)} \wedge b_n$. Then, for any $c_n \geq 1$, 
\begin{align*}
&\lfloor s_n \rfloor \mathbb{E}_n \left[   1(W_{Y_1} > b_n) + \mathcal{P}_n(Y_1 ) +  \bar W_{Y_1} \eta_n  + \frac{b_n {\bar W}_{Y_1}}{\theta n} s_n   \right] \\
&\leq s_n \sum_{j=1}^n \frac{\bar W_j}{\Lambda_n} \left(  1(W_j > b_n) + \mathcal{P}_n(j ) +  \bar W_{j} \eta_n  + \frac{b_n {\bar W}_{j}}{\theta n} s_n \right) \\
&\leq \frac{s_n n}{\Lambda_n} \cdot \frac{1}{n} \sum_{j=1}^n (\bar W_j - c_n)^+ + \frac{s_n n}{\Lambda_n} \cdot \frac{1}{n} \sum_{j=1}^n c_n \left( 1(W_j > b_n) + \mathcal{P}_n(j ) +  \bar W_{j} \eta_n  + \frac{b_n {\bar W}_{j}}{\theta n} s_n \right) \\
&= \frac{s_n}{\mathbb{E}_n[ \bar W^{(n)}] } \left(  \mathbb{E}_n\left[ (\bar W^{(n)} - c_n)^+ \right] + c_n \mathbb{P}_n(W^{(n)} > b_n)  + c_n \mathcal{E}_n +   \mathbb{E}_n\left[ \bar W^{(n)} \right] \left(c_n \eta_n + \frac{c_n b_n s_n}{\theta n} \right) \right) \\
&= O_P\left( s_n \left( g(c_n) + \Delta_n + c_n P(W > b_n) + c_n \mathcal{E}_n + c_n \eta_n + \frac{c_n b_n s_n}{n} \right)  \right),
\end{align*}
as $n \to \infty$, and since $\eta_n = O_P\left( \Delta_n + g(b_n) + b_n^2/n \right)$, we conclude that
\begin{align*}
\mathbb{P}_{n,i}(\tau = r, M_r ) &= O_P\left( s_n \left( g(c_n)   + c_n \mathcal{E}_n +  c_n  \Delta_n + c_n g(b_n) +  c_n b_n^2/n \right) \right),
\end{align*}
as $n \to \infty$. It now follows from the beginning of the proof that
\begin{align*}
\mathbb{P}_n( \tau \leq k ) &\leq O_P\left( k s_n \left( g(c_n)   + c_n \mathcal{E}_n +  c_n  \Delta_n + c_n g(b_n) +  c_n b_n^2/n \right) \right) \\
&\hspace{5mm} + P\left( |\mathcal{V}_{k+1}| > s_n \right) + \mathbb{P}_n\left( \hat T^{(k)} \not\simeq \mathcal{T}^{(k)} \right) ,
\end{align*}
as $n \to \infty$. Since $\lim_{x \to \infty} g(x) = 0$, choosing, for example, $c_n = \left( \mathcal{E}_n + \Delta_n + g(b_n) + b_n^2/n \right)^{-1/2}$ and $s_n = (g(c_n) + c_n^{-1/2})^{-1/2}$ proves the theorem.
\end{proof}

\subsection{Discrete coupling for directed graphs} \label{SS.CouplingDirected}

The equivalent of Theorem~\ref{T.Coupling} ($m=1$) for directed graphs has already been proven, under conditions equivalent to those in Assumption~\ref{A.PrimitivesD}, in \cite{Olvera_20}  (Theorem~6.3) for the DCM, and in \cite{Lee_Olv_20} (Theorem~3.7) for the IRD. Hence, we only need to describe the distribution of the intermediate tree and state the coupling theorem. The descriptions of the couplings follow, with some adjustments, those from Sections~\ref{S.CouplingCM} and \ref{S.CouplingIR}. However, the precise descriptions in the directed case can be found in \cite{Chen_Lit_Olv_17} (Section~5.2) for the DCM and in \cite{Lee_Olv_20} (Section~3.2.2) for the IRD. 

In the directed case, the intermediate tree $\hat T$ is constructed using a sequence of conditionally independent  (given $\mathscr{F}_n$) random vectors $\{ (\hat N_\bi, \hat D_\bi, \mathbf{\hat A}_\bi): \bi \in \mathcal{U}\}$ in $\mathcal{S}$, with $\{ (\hat N_\bi, \hat D_\bi, \mathbf{\hat A}_\bi): \bi \in \mathcal{U}, \bi \neq \emptyset \}$ conditionally i.i.d. The tree $\hat T$ is constructed as in the undirected case using the $\{\hat N_\bi\}$, with all edges pointing towards the root, and the full marks take the form:
$$\mathbf{\hat X}_\bi = (\hat N_\bi, \hat D_\bi,  \mathbf{\hat A}_\bi), \quad \bi \in \mathcal{U}.$$
The marked tree is given by $\hat T(\mathbf{\hat A}) = \{ \mathbf{\hat X}_\bi: \bi \in \hat T \}$.

We now specify the distribution of the full marks, which in the case of a DCM is given by: 
\begin{align*}
\mathbb{P}_n\left( \mathbf{\hat X}_\emptyset \in \cdot \right) &= \frac{1}{n} \sum_{i=1}^n 1((D_i^-, D_i^+, {\bf a}_i) \in \cdot),  \\
\mathbb{P}_n\left( \mathbf{\hat X}_\bi \in \cdot \right) &=  \sum_{i=1}^n \frac{D_i^+}{L_n}1((D_i^-, D_i^+, {\bf a}_i) \in \cdot), \quad \bi\neq \emptyset .
\end{align*}
For the IRD model, first let $\{a_n\}$ and $\{b_n\}$ be a sequences such that $a_n \wedge b_n \xrightarrow{P} \infty$ and $a_n b_n/n \xrightarrow{P} 0$ as $n\to\infty$, and use them to define $\bar W_i^- = W_i^- \wedge a_n$ and $\bar W_i^+ = W_i^+ \wedge b_n$, 
$$\Lambda_n^- = \sum_{i=1}^n \bar W_i^- \qquad \text{and} \qquad \Lambda_n^+ = \sum_{i=1}^n \bar W_i^+.$$
The marks on the coupled marked Galton-Watson process are given by:
\begin{align*}
\mathbb{P}_n\left( \mathbf{\hat X}_\emptyset \in \cdot \right) &= \frac{1}{n} \sum_{i=1}^n P( (D_i^-, D_i^+, {\bf a}_i) \in \cdot | {\bf a}_i),  \\
\mathbb{P}_n\left( \mathbf{\hat X}_\bi \in \cdot \right) &=  \sum_{i=1}^n \frac{\bar W_i^+}{\Lambda_n^+} P( (D_i^-, D_i^++1, {\bf a}_i) \in \cdot | {\bf a}_i ) , \quad \bi \neq \emptyset,
\end{align*}
where conditionally on ${\bf a}_i$, $D_i^-$ and $D_i^+$ are independent Poisson random variables with means $\Lambda_n^+ \bar W_i^+/(\theta n)$ and $\Lambda_n^- \bar W_i^-/(\theta n)$, respectively.

The intermediate coupling theorem for directed graphs is given below, and it is a direct consequence of Theorem~6.3 in \cite{Olvera_20} and Theorem~3.7 in \cite{Lee_Olv_20}. 

\begin{theo} \label{T.CouplingD}
Suppose $G(V_n, E_n)$ is either a DCM or an IRD satisfying Assumption~\ref{A.PrimitivesD}. Then, for $\mathcal{G}^{(k)}_I({\bf a})$ the depth-$k$ neighborhood of a uniformly chosen vertex $I \in V_n$, there exists a marked Galton-Watson tree $\hat T^{(k)} (\mathbf{\hat A})$ restricted to its first $k$ generations, whose root corresponds to vertex $I$, and 
such that for any fixed $k \geq 1$,
$$\mathbb{P}_n\left( \mathcal{G}_I^{(k)}(\mathbf{a}) \not\simeq \hat T^{(k)}(\mathbf{\hat A}) \right) \xrightarrow{P} 0, \qquad n \to \infty.$$
\end{theo}

\subsection{Coupling between two trees}

In view of Theorems~\ref{T.Coupling} and \ref{T.CouplingD}, the proofs of the main theorems, Theorem~\ref{T.MainU} and \ref{T.MainD}, ($m = 1$) will be complete once we establish that with high probability the intermediate tree $\hat T^{(k)}$ is isomorphic to the limiting tree $\mathcal{T}^{(k)}$, and that the node marks in the two trees are within $\epsilon$ distance of each other. 

{\bf Note:} There is no need to consider the undirected and directed cases separately, since they only differ on the sample space for the full marks, $\mathbf{\hat X}_\bi$ / $\boldsymbol{X}_\bi$,  which take values in $\mathcal{S} = \mathbb{N} \times \mathbb{R} \times \mathcal{S}'$ in the undirected case and $\mathcal{S} = \mathbb{N} \times \mathbb{N} \times \mathbb{R} \times \mathbb{R} \times \mathcal{S}'$ in the directed one. For the directed case, all edges in the trees point towards the root. 

The coupling theorem between the two trees is the following. The proof of the main theorems, Theorems~\ref{T.MainU} and \ref{T.MainD}, will follow directly from combining Theorems~\ref{T.Coupling} and \ref{T.TreeToTree} in the undirected case, and Theorems~\ref{T.CouplingD} and \ref{T.TreeToTree} in the directed one. 

\begin{theo} \label{T.TreeToTree}
Under Assumption~\ref{A.PrimitivesU} or \ref{A.PrimitivesD}, as appropriate, there exists a coupling of $\hat T^{(k)}(\mathbf{\hat A})$ and $\mathcal{T}^{(k)}(\boldsymbol{A})$ such that
$$\mathbb{P}_n \left( \hat T^{(k)} \not\simeq \mathcal{T}^{(k)} \right) \xrightarrow{P} 0, \qquad n \to \infty,$$
and such that for any $\epsilon > 0$,
$$\mathbb{E}_n\left[ \rho(\mathbf{\hat X}_\emptyset, \boldsymbol{X}_\emptyset) \right] \xrightarrow{P} 0 \qquad \text{and} \qquad \mathbb{P}_n \left(  \bigcap_{{\bf i} \in  \mathcal{T}^{(k)}} \{ \rho(\mathbf{\hat X}_{\bf i}, \boldsymbol{X}_{\bf i}) \leq \epsilon \} , \, \hat T^{(k)} \simeq \mathcal{T}^{(k)} \right)  \xrightarrow{P} 1, \qquad n \to \infty.$$
\end{theo}

Before proving Theorem~\ref{T.TreeToTree}, we will need to prove a couple of technical lemmas. The first of the two establishes the existence of couplings for the node attributes, whose distributions are given by:
\begin{align*}
\nu_n(\cdot) &= \mathbb{P}_n\left( \mathbf{\hat A} \in \cdot \right) = \frac{1}{n} \sum_{i=1}^n 1(\mathbf{a}_i \in \cdot ) \qquad \text{and} \qquad \nu(\cdot) = P\left(\boldsymbol{A} \in \cdot \right),
\end{align*}
and their size-biased versions. Recall that in the undirected case the node attributes are of the form $\mathbf{a}_i = (D_i, \mathbf{b}_i)$ in the CM and $\mathbf{a}_i = (W_i, \mathbf{b}_i)$ in the IR, while in the directed case they take the form $\mathbf{a}_i = (D_i^-, D_i^+, \mathbf{b}_i)$ in the DCM and $\mathbf{a}_i = (W_i^-, W_i^+, \mathbf{b}_i)$ in the IRD. In the undirected case, the size-bias is done with respect to the first coordinate, while in the directed case with respect to the second one. Specifically, the size-biased attributes in the undirected case take the form:
\begin{align*}
\mathbb{P}_n\left( \mathbf{\hat A}_b \in \cdot \right) &= \begin{cases}
L_n^{-1} \sum_{i=1}^n D_i 1( (D_i, \mathbf{b}_i) \in \cdot) , & \text{in the CM},\\
\Lambda_n^{-1} \sum_{i=1}^n \bar W_i 1( (W_i, \mathbf{b}_i) \in \cdot), & \text{in the IR},
\end{cases}
\end{align*}
and 
\begin{align*}
P\left( \boldsymbol{A}_b \in \cdot \right) &= \begin{cases}
 E[ \mathscr{D} 1( (D, \boldsymbol{B}) \in \cdot) ] /E[\mathscr{D}], & \text{in the CM}, \\
E[ W 1( (W, \boldsymbol{B}) \in \cdot) ]/ E[W], & \text{in the IR},
\end{cases}
\end{align*}
while in the directed case they take the form:
\begin{align*}
\mathbb{P}_n\left( \mathbf{\hat A}_b \in \cdot \right) &= \begin{cases}
L_n^{-1} \sum_{i=1}^n D_i^+ 1( (D_i^-, D_i^+, \mathbf{b}_i) \in \cdot) , & \text{in the DCM},\\
(\Lambda_n^+)^{-1} \sum_{i=1}^n \bar W_i^+ 1( (W_i^-, W_i^+, \mathbf{b}_i) \in \cdot), & \text{in the IRD},
\end{cases}
\end{align*}
and 
\begin{align*}
P\left( \boldsymbol{A}_b \in \cdot \right) &= \begin{cases}
 E[ \mathscr{D}^+ 1((\mathscr{D}^-, \mathscr{D}^+, \boldsymbol{B}) \in \cdot) ] /E[\mathscr{D}^+], & \text{in the DCM}, \\
E[ W^+ 1( (W^-, W^+, \boldsymbol{B}) \in \cdot) ]/ E[W^+], & \text{in the IRD}.
\end{cases}
\end{align*}

For the undirected case, let $\rho''$ be the metric on $\mathcal{S}'' = [0,\infty) \times \mathcal{S}'$ given by
$$\rho''(\mathbf{x}, \mathbf{y}) = |x_1 - y_1| + \rho'(\mathbf{x}_2, \mathbf{y}_2), \qquad \mathbf{x} = (x_1, \mathbf{x}_2), \, \mathbf{y} = (y_1, \mathbf{y}_2),$$
and for the directed case let $\rho''$ be the metric on $\mathcal{S}'' = [0,\infty) \times [0,\infty) \times \mathcal{S}'$ given by
$$\rho''(\mathbf{x}, \mathbf{y}) = |x_1 - y_1| + |x_1-y_2| +  \rho'(\mathbf{x}_3, \mathbf{y}_3), \qquad \mathbf{x} = (x_1, x_2, \mathbf{x}_3), \, \mathbf{y} = (y_1,y_2,  \mathbf{y}_3).$$

\begin{lemma} \label{L.AttributesConv}
Under Assumption~\ref{A.PrimitivesU} or \ref{A.PrimitivesD}, as appropriate, there exist couplings $(\mathbf{\hat A}, \boldsymbol{A})$ and $(\mathbf{\hat A}_b, \boldsymbol{A}_b)$ constructed on the same probability space $(\mathcal{S}'', \mathscr{F}_n, \mathbb{P}_n)$ such that
$$\mathbb{E}_n\left[ \rho''( \mathbf{\hat A}, \boldsymbol{A}) \right] \xrightarrow{P} 0, \qquad \rho''( \mathbf{\hat A}, \boldsymbol{A})  \xrightarrow{P}  0  \qquad \text{and} \qquad \rho''( \mathbf{\hat A}_b, \boldsymbol{A}_b)  \xrightarrow{P} 0 $$
as $n \to \infty$. 
\end{lemma}

\begin{proof}
Assumptions~\ref{A.PrimitivesU} and \ref{A.PrimitivesD} state that $W_1(\nu_n, \nu) \xrightarrow{P} 0$ as $n \to \infty$, and by the properties of the Wasserstein metric (see Theorem~4.1 in \cite{Villani_2009}), there exists an optimal coupling $(\mathbf{\hat A}, \boldsymbol{A})$ such that 
\begin{equation*} 
\mathbb{E}_n\left[ \rho''( \mathbf{\hat A}, \boldsymbol{A}) \right] = W_1(\nu_n, \nu) \xrightarrow{P} 0 \qquad \text{and} \qquad \rho''( \mathbf{\hat A}, \boldsymbol{A})  \xrightarrow{P}  0
\end{equation*}
as $n \to \infty$. 

For the biased versions, note that it suffices to prove the lemma for the undirected case, since a simple rearrangement of terms:
$$\mathbf{a}_i' = (D_i', \mathbf{b}_i') := (D_i^+, D_i^-, \mathbf{b}_i) \qquad \text{or} \qquad \mathbf{a}_i' = (W_i', \mathbf{b}_i') := (W_i^+, W_i^-, \mathbf{b}_i)$$
reduces the directed case to the undirected one. Through the remainder of the proof, write $\mathbf{\hat A} = (\hat Y, \mathbf{\hat B})$ and $\boldsymbol{A} = (Y, \boldsymbol{B})$ to avoid having to separate the CM and IR cases. 

Next, note that we only need to show that that $\mathbf{\hat A}_b \Rightarrow \boldsymbol{A}_b$ as $n \to \infty$, where $\Rightarrow$ denotes convergence in distribution, since then we can take the almost sure representation to obtain that $\rho''(\mathbf{\hat A}_b, \boldsymbol{A}_b) \xrightarrow{P} 0$. To this end, let $f: \mathcal{S}'' \to \mathbb{R}$ be a bounded and continuous function, and let $(\mathbf{\hat A}, \boldsymbol{A})$ be the one from the beginning of the proof. Let $\tilde Y = \hat Y$ if the graph is a CM or $\tilde Y = \hat Y \wedge b_n$ if it is an IR. Then,  
\begin{align*}
\left| \mathbb{E}_n\left[ f(\mathbf{\hat A}_b) \right] - E[f(\boldsymbol{A}_b)] \right| &= \left| \frac{1}{\mathbb{E}_n[ \tilde Y]} \mathbb{E}_n\left[ \tilde Y f( \mathbf{\hat A} ) \right] - \frac{1}{E[Y]} E\left[ Y f(\boldsymbol{A}) \right] \right| \\
&\leq  \frac{1}{\mathbb{E}_n[\tilde Y]} \left( \left| \mathbb{E}_n\left[ (\tilde Y - Y) f(\mathbf{\hat A}) \right] \right| + \left| \mathbb{E}_n \left[ Y (f(\mathbf{\hat A}) - f(\boldsymbol{A})) \right] \right|  \right) \\
&\hspace{5mm} +  \left| \frac{1}{\mathbb{E}_n[\tilde Y]} - \frac{1}{E[Y]} \right| \left| E[Y f(\boldsymbol{A})] \right| \\
&\leq  \frac{1}{\mathbb{E}_n[\tilde Y]} \left(  \mathbb{E}_n\left[ |\tilde Y - Y|  \right] \sup_{\mathbf{a} \in \mathcal{S}''} |f(\mathbf{a})| + \left| \mathbb{E}_n \left[ Y (f(\mathbf{\hat A}) - f(\boldsymbol{A})) \right] \right|  \right) \\
&\hspace{5mm}  +  \frac{\mathbb{E}_n[ |\tilde Y - Y| [}{\mathbb{E}_n[\tilde Y]E[Y]}  \left| E[Y f(\boldsymbol{A})] \right| \\
&\leq  \frac{1}{\mathbb{E}_n[\tilde Y]} \left(  \mathbb{E}_n\left[ |\tilde Y - Y|  \right] 2 \sup_{\mathbf{a} \in \mathcal{S}''} |f(\mathbf{a})| +  \mathbb{E}_n \left[ Y |f(\mathbf{\hat A}) - f(\boldsymbol{A})| \right]  \right).
\end{align*}
Since $W_1(\nu_n, \nu) \xrightarrow{P} 0$ implies that $\mathbb{E}_n[ |\hat Y - Y| ] \xrightarrow{P} 0$ as $n \to \infty$, we have
$$ \mathbb{E}_n\left[ |\tilde Y - Y|  \right] \leq  \mathbb{E}_n\left[ |\hat Y - Y |  \right] +  E\left[ Y 1(Y > b_n) \right] \xrightarrow{P} 0$$
and $\mathbb{E}_n [\tilde Y] \xrightarrow{P} E[Y]$ as $n \to \infty$. And by the dominated convergence theorem,
$$\lim_{n \to \infty} E\left[ \mathbb{E}_n \left[ Y |f(\mathbf{\hat A}) - f(\boldsymbol{A})| \right]  \right] = E\left[ \lim_{n \to \infty} Y |f(\mathbf{\hat A}) - f(\boldsymbol{A})| \right] = 0.$$
Hence, $ \mathbb{E}_n \left[ Y |f(\mathbf{\hat A}) - f(\boldsymbol{A})| \right]  \xrightarrow{P} 0$ as $n \to \infty$, and $\mathbf{A}_b \Rightarrow \boldsymbol{A}_b$ as required. 
\end{proof}

The second technical lemma relates the convergence of the attributes to that of the full marks. 

\begin{lemma} \label{L.WeakConvergence}
Suppose Assumption~\ref{A.PrimitivesU} or \ref{A.PrimitivesD} holds, as appropriate, and let  $(\mathbf{\hat A}, \boldsymbol{A})$ and $(\mathbf{\hat A}_b, \boldsymbol{A}_b)$ be the couplings in Lemma~\ref{L.AttributesConv}. Then, there exist couplings  for $(\mathbf{\hat X}_\emptyset, \boldsymbol{X}_0)$ and $(\mathbf{\hat X}, \boldsymbol{X})$ constructed on the same probability space as $(\mathbf{\hat A}, \boldsymbol{A})$ and $(\mathbf{\hat A}_b, \boldsymbol{A}_b)$, such that 
$$\mathbb{E}_n\left[ \rho( \mathbf{\hat X}_\emptyset, \boldsymbol{X}_0) \right] \xrightarrow{P} 0, \qquad \rho( \mathbf{\hat X}_\emptyset, \boldsymbol{X}_0)  \xrightarrow{P}  0  \qquad \text{and} \qquad \rho( \mathbf{\hat X}, \boldsymbol{X})  \xrightarrow{P} 0 $$
as $n \to \infty$. 
\end{lemma}

\begin{proof}
For the two undirected models, CM and IR, write:
\begin{align*}
\mathbf{\hat A} &= (\hat Y, \mathbf{\hat B}), \hspace{11mm} \boldsymbol{A} = (Y, \boldsymbol{B}) \\
\mathbf{\hat A}_b &= (\hat Y_b, \mathbf{\hat B}_b), \qquad \boldsymbol{A}_b = (Y_b, \boldsymbol{B}_b) .
\end{align*}
For the two directed models, DCM and IRD, write:
\begin{align*}
\mathbf{\hat A} &= (\hat Y^-, \hat Y^+, \mathbf{\hat B}), \hspace{11mm} \boldsymbol{A} = (Y^-, Y^+, \boldsymbol{B}) \\
\mathbf{\hat A}_b &= (\hat Y_b^-, \hat Y_b^+, \mathbf{\hat B}_b), \qquad \boldsymbol{A}_b = (Y_b^-, Y_b^+, \boldsymbol{B}_b) .
\end{align*}

To obtain the statement of the lemma for the CM, simply set $(\mathbf{\hat X}_\emptyset, \boldsymbol{X}_\emptyset) = (\hat Y, \mathbf{\hat A}, Y, \boldsymbol{A})$ and $(\mathbf{\hat X}_1, \boldsymbol{X}_1) = (\hat Y_b, \mathbf{\hat A}_b, Y_b, \boldsymbol{A}_b)$. Similarly, for the DCM set $(\mathbf{\hat X}_\emptyset, \boldsymbol{X}_\emptyset) = (\hat Y^-, \hat Y^+, \mathbf{\hat A}, Y^-, Y^+, \boldsymbol{A})$ and $(\mathbf{\hat X}_1, \boldsymbol{X}_1) = (\hat Y_b^-, \hat Y_b^+, \mathbf{\hat A}_b, Y_b^-, Y_b^+, \boldsymbol{A}_b)$.

For the IR construct 
\begin{align*}
(\hat S, S) &= \left( \Lambda_n (\hat Y \wedge b_n)/(\theta n), Y \right) .
\end{align*}
Note that our assumptions imply that $\mathbb{E}_n \left[ |\hat S - S| \right] \xrightarrow{P} 0$ as $n \to \infty$. Now let $U\sim$ Uniform$[0,1]$ be i.i.d.~and independent of $(\hat S, S)$, and take
$$(\mathbf{\hat X}_\emptyset, \boldsymbol{X}_0 ) = \left( G^{-1}(U; \hat S), \mathbf{\hat A}, G^{-1}(U; Y), \boldsymbol{A} \right),$$
where $G^{-1}(u;\lambda) = \sum_{m=0}^\infty m 1(G(m; \lambda) \leq u < G(m+1;\lambda))$ is the generalized inverse of the Poisson distribution function with mean $\lambda$. Note that since $G(m;\lambda)$ is decreasing in $\lambda$ for all $m \geq 0$, then we have that $\text{Poi}(\lambda) \geq_{\text{s.t.}} \text{Poi}(\mu)$ whenever $\lambda \geq \mu$, where $\geq_{\text{s.t.}}$ denotes the usual stochastic order and $\text{Poi}(\alpha)$ denotes a Poisson random variable with mean $\alpha$. It follows that $E\left[\left| G^{-1}(U;\lambda)- G^{-1}(U;\mu) \right|\right| = |\lambda-\mu|$, which in turn implies that 
$$\mathbb{E}_n\left[ \rho( \mathbf{\hat X}_\emptyset, \boldsymbol{X}_0) \right] = \mathbb{E}_n\left[ |\hat S - S| + \rho''(\mathbf{\hat A}, \boldsymbol{A}) \right] \xrightarrow{P} 0, \qquad  n \to \infty.$$

For the size-biased versions, set
\begin{align*}
( \hat S_b, S_b) &= \left( \Lambda_n (\hat Y_b \wedge b_n)/(\theta n), \, Y_b \right) ,
\end{align*}
note that Lemma~\ref{L.AttributesConv} gives $| \hat S_b - S_b| \xrightarrow{P} 0$ as $n \to \infty$, and let
$$(\mathbf{\hat X}_1, \boldsymbol{X}_1 ) = \left( G^{-1}(U; \hat S_b)+1, \mathbf{\hat A}_b, G^{-1}(U; Y_b)+1, \boldsymbol{A}_b \right).$$
Now use the continuity in $\lambda$ of $G^{-1}(u; \lambda)$ to obtain that 
$$\rho( \mathbf{\hat X}_1, \boldsymbol{X}_1 ) = \left| G^{-1}(U; \hat S_b) - G^{-1}(U; Y_b) \right| + \rho''( \mathbf{\hat A}_b, \boldsymbol{A}_b) \xrightarrow{P} 0, \qquad n \to \infty.$$

The same steps also give the result for the IRD by setting:
\begin{align*}
( \hat S^-, \hat S^+, S^-, S^+) &= \left( \Lambda_n^+ (\hat Y^- \wedge a_n)/(\theta n), \, \Lambda_n^- (\hat Y^+ \wedge b_n)/(\theta n), \, cY^-, (1-c) Y^+ \right), \\
( \hat S_b^-, \hat S_b^+, S_b^-, S_b^+) &= \left( \Lambda_n^+ (\hat Y_b^- \wedge a_n)/(\theta n), \, \Lambda_n^- (\hat Y_b^+ \wedge b_n)/(\theta n), \, cY_b^-, (1-c) Y_b^+ \right),
\end{align*}
where $c = E[W^+]/E[W^- + W^+]$, and setting
\begin{align*}
(\mathbf{\hat X}_\emptyset, \boldsymbol{X}_\emptyset ) &= \left( G^{-1}(U; \hat S^-), \, G^{-1}(U'; \hat S^+)+1, \,  \mathbf{\hat A}, \, G^{-1}(U; S^-), \, G^{-1}(U'; S^+) +1, \, \boldsymbol{A} \right), \\
(\mathbf{\hat X}_1, \boldsymbol{X}_1 ) &= \left( G^{-1}(U; \hat S_b^-), \, G^{-1}(U'; \hat S_b^+)+1, \, \mathbf{\hat A}_b, \, G^{-1}(U'; S_b^-), \, G^{-1}(U; S_b^+) +1, \, \boldsymbol{A}_b \right), 
\end{align*}
for some $U, U'$ i.i.d.~Uniform$[0,1]$ and independent of $\mathscr{F}_n$. This completes the proof. 
\end{proof}

Finally, we can give the proof of Theorem~\ref{T.TreeToTree}.

\begin{proof}[Proof of Theorem~\ref{T.TreeToTree}]
By Lemma~\ref{L.WeakConvergence} there exists couplings $(\mathbf{\hat X}_\emptyset, \boldsymbol{X}_\emptyset)$ and $(\mathbf{\hat X}_1, \boldsymbol{X}_1)$ such that
$$\mathbb{E}_n\left[ \rho( \mathbf{\hat X}_\emptyset, \boldsymbol{X}_\emptyset) \right] \xrightarrow{P} 0 \qquad \text{and} \qquad \rho(\mathbf{\hat X}_1, \boldsymbol{X}_1) \xrightarrow{P} 0,$$
as $n \to \infty$. Now let $\{ (\mathbf{\hat X}_\bi, \boldsymbol{X}_\bi): \bi \in \mathcal{U}, \bi \neq \emptyset\}$ be i.i.d.~copies of $(\mathbf{\hat X}_1, \boldsymbol{X}_1)$, independent of  $(\mathbf{\hat X}_\emptyset, \boldsymbol{X}_\emptyset)$. Recall that $\hat N_\bi$ ($\mathcal{N}_\bi$) can be determined from the first coordinate of $\mathbf{\hat X}_\bi$ ($\boldsymbol{X}_\bi$).

We will now use the sequence $\{ (\mathbf{\hat X}_\bi, \boldsymbol{X}_\bi): \bi \in \mathcal{U}\}$ to construct both $\hat T(\mathbf{\hat A})$ and $\mathcal{T}(\boldsymbol{A})$ by determining their nodes according to the recursions:
$$\hat A_k = \{ (\bi, j): \bi \in \hat A_{k-1}, 1 \leq j \leq \hat N_\bi \} \qquad \text{and} \qquad \mathcal{A}_k = \{ (\bi,j): \bi \in \mathcal{A}_{k-1}, 1 \leq j \leq \mathcal{N}_\bi \},$$
for $k \geq 1$. Without loss of generality assume that $0 < \epsilon < 1$.

Now define the stopping time
$$\kappa(\epsilon) = \inf\left\{ k \geq 0: \rho(\mathbf{\hat X}_\bi, \boldsymbol{X}_\bi) > \epsilon \text{ for some } \bi \in \hat A_k \right\}.$$
Note that since $\hat N_\bi$ and $\mathcal{N}_\bi$ are integer-valued, then $\rho (\mathbf{\hat X}_\bi, \boldsymbol{X}_\bi) < 1$ implies that $\hat N_\bi = \mathcal{N}_\bi$. It follows that for any $x_n \geq 1$, 
\begin{align*}
\mathbb{P}_n \left( \hat T^{(k)} \simeq \mathcal{T}^{(k)} \right) &\geq \mathbb{P}_n \left( \bigcap_{r=0}^k \bigcap_{{\bf i} \in  \mathcal{A}_r} \{ \rho(\mathbf{\hat X}_{\bf i}, \boldsymbol{X}_{\bf i}) \leq \epsilon \} , \, \hat T^{(k)} \simeq \mathcal{T}^{(k)} \right)  \\
&= \mathbb{P}_n\left( \kappa(\epsilon) > k \right) \\
&\geq 1  -  \mathbb{P}_n\left( \kappa(\epsilon) \leq k, |\mathcal{V}_k| \leq x_n \right)  -  \mathbb{P}_n\left( |\mathcal{V}_k| > x_n \right)  \\
&= 1 - \sum_{r=0}^k \mathbb{P}_n\left( \kappa(\epsilon) = r, |\mathcal{V}_k| \leq x_n \right)  -  \mathbb{P}_n\left( |\mathcal{V}_k| > x_n \right),
\end{align*}
where $\mathcal{V}_k = \bigcup_{r=0}^k \mathcal{A}_r$. To compute the last probabilities, note that $\mathbb{P}_n(\kappa(\epsilon) = 0) \leq \epsilon^{-1} \mathbb{E}_n\left[ \rho( \mathbf{\hat X}_\emptyset, \boldsymbol{X}_\emptyset) \right]$, and for $r \geq 1$:
\begin{align*}
\mathbb{P}_n\left( \kappa(\epsilon) = r , |\mathcal{V}_k| \leq x_n \right) &\leq \mathbb{P}_n\left(  \bigcup_{\bi \in \mathcal{A}_{r}} \{  \rho(\mathbf{\hat X}_{\bf i}, \boldsymbol{X}_{\bf i}) > \epsilon\} , |\mathcal{A}_{r} |\leq x_n \right) \\
&\leq \mathbb{E}_n\left[ 1( |\mathcal{A}_r| \leq x_n) \sum_{\bi \in \mathcal{A}_r} 1\left(  \rho(\mathbf{\hat X}_{\bf i}, \boldsymbol{X}_{\bf i}) > \epsilon \right) \right] \\
&=  \mathbb{E}_n\left[ 1( |\mathcal{A}_r| \leq x_n) | \mathcal{A}_r| \right] \mathbb{P}_n \left(  \rho(\mathbf{\hat X}_1, \boldsymbol{X}_1) > \epsilon \right) \\
&\leq x_n  \mathbb{P}_n \left(  \rho(\mathbf{\hat X}_1, \boldsymbol{X}_1) > \epsilon \right),
\end{align*}
where in the third step we used the independence of $(\mathbf{\hat X}_1, \boldsymbol{X}_1)$ from $\mathcal{A}_r$. It follows that if we choose $x_n = \mathbb{P}_n \left(  \rho(\mathbf{\hat X}_1, \boldsymbol{X}_1) > \epsilon \right)^{-1/2} \xrightarrow{P} \infty$, then
\begin{align*}
&\mathbb{P}_n \left( \bigcap_{r=0}^k \bigcap_{{\bf i} \in  \mathcal{A}_r} \{ \rho(\mathbf{\hat X}_{\bf i}, \boldsymbol{X}_{\bf i}) \leq \epsilon \} , \, \hat T^{(k)} \simeq \mathcal{T}^{(k)} \right)  \\
&\geq 1 - \epsilon^{-1}  \mathbb{E}_n\left[ \rho( \mathbf{\hat X}_\emptyset, \boldsymbol{X}_\emptyset) \right] - k x_n \mathbb{P}_n \left(  \rho(\mathbf{\hat X}_1, \boldsymbol{X}_1) > \epsilon \right) -   \mathbb{P}_n\left( |\mathcal{V}_k| > x_n \right) \\
&\geq 1 - \epsilon^{-1}  \mathbb{E}_n\left[ \rho( \mathbf{\hat X}_\emptyset, \boldsymbol{X}_\emptyset) \right] - k x_n^{-1/2} -   \mathbb{P}_n\left( |\mathcal{V}_k| > x_n \right) \xrightarrow{P} 0, 
\end{align*}
as $n \to \infty$. This completes the proof.
\end{proof}

The last proof in the paper relates to the case $m \geq 2$ for both Theorems~\ref{T.MainU} and \ref{T.MainD}. Since the proof for the directed case follows exactly the same steps as for the undirected one, we include here only the undirected case.

\begin{proof}[Proof of Theorem~\ref{T.MainU} $(m \geq 2)$]
Start by sampling $\{I_j: 1 \leq j \leq m\}$ independently and uniformly in $V_n$. Without loss of generality we can assume that $I_1 \neq I_2 \neq \dots \neq I_m$. Next, note that the couplings for $\mathcal{G}_i^{(k)}(\mathbf{a})$ with their corresponding intermediate trees can be done simultaneously for all $i \in V_n$, so let $\hat T^{(k)}_{\emptyset(i)}(\mathbf{\hat A})$ be the coupled tree for $\mathcal{G}_i^{(k)}(\mathbf{a})$. Note that the $\{ \hat T^{(k)}_{\emptyset(I_j)}(\mathbf{\hat A}): 1 \leq j \leq m\}$ are not independent of each other, but by Theorem~\ref{T.MainU} $(m = 1)$ case, they satisfy
$$\sum_{j=1}^m \mathbb{E}_n\left[ \rho(\mathbf{X}_{I_j}, \mathbf{\hat X}_{\emptyset(I_j)}) \right] \xrightarrow{P} 0 \qquad \text{and} \qquad \mathbb{P}_n\left( \bigcup_{j=1}^m \left\{ \mathcal{G}_{I_j}^{(k)} (\mathbf{a}) \not\simeq \hat T^{(k)}_{\emptyset(I_j)} (\mathbf{\hat A})  \right\}  \right) \xrightarrow{P} 0,$$
as $n \to \infty$. We will now explain how to construct a set of i.i.d.~copies of $\hat T_{\emptyset(I_1)}^{(k)} (\mathbf{\hat A})$, denoted $\left\{ \tilde T_{\emptyset(I_j)}^{(k)}(\mathbf{\tilde A}) : 1 \leq j \leq m\right\}$, satisfying
\begin{equation} \label{eq:IndependentCopies}
\mathbb{P}_n\left( \bigcup_{j=1}^m \left\{  \hat T^{(k)}_{\emptyset(I_j)} (\mathbf{\hat A}) \not\simeq \tilde T^{(k)}_{\emptyset(I_j)} (\mathbf{\tilde A})  \right\}  \right) \xrightarrow{P} 0, \qquad n \to \infty.
\end{equation}

To start, let $\tilde T_{\emptyset(I_1)}^{(k)}(\mathbf{\tilde A}) = \hat T_{\emptyset(I_1)}^{(k)}(\mathbf{\hat A})$. Next, note that the trees $\left\{ \hat T_{\emptyset(I_j)}^{(k)}(\mathbf{\hat A}): 1 \leq j \leq m \right\}$ are each (delayed) marked Galton-Watson processes whose roots have distribution $\mu_n^*(\cdot) = \mathbb{P}_n\left( \mathbf{\hat X}_\emptyset \in \cdot \right)$ and all other nodes have distribution $\mu_n(\cdot) = \mathbb{P}_n\left( \mathbf{\hat X}_1 \in \cdot \right)$. Since each full mark $\mathbf{\hat X}_\mathbf{i}$ contains the vector $\mathbf{\hat A}_\mathbf{i}$, we can see which vertex attributes have been sampled. Note that the possible vertex attributes are $\{\mathbf{a}_1, \dots, \mathbf{a}_n\}$, and they are such that $p_i^* = \mathbb{P}_n( \mathbf{\hat A}_\emptyset = \mathbf{a}_i) = 1/n$, and $p_i = \mathbb{P}_n(\mathbf{\hat A}_1 = \mathbf{a}_i)$ is either equal to $D_i/L_n$ in the CM or to $\bar W_i/\Lambda_n$ in the IR. We will keep track of the labels $\{1, 2, \dots, n\}$ of the vertex attributes $\{\mathbf{a}_1, \mathbf{a}_2, \dots, \mathbf{a}_n\}$. To do this, let $S$ be the set of labels sampled in the construction of $\tilde T_{\emptyset(I_1)}^{(k)}(\mathbf{\tilde A})$, and then construct each of the $\tilde T_{\emptyset(I_j)}^{(k)}(\mathbf{\tilde A})$, $2 \leq j \leq m$, in a breadth-first fashion according to the following rule:
\begin{itemize}
\item If a node in $\hat T_{\emptyset(I_j)}^{(k)}(\mathbf{\hat A})$ has a vertex attribute whose label is not in the set $S$, copy the node onto $\tilde T_{\emptyset(I_j)}^{(k)}(\mathbf{\tilde A})$ and add the new observed label to the set $S$.
\item Otherwise, attach an independent copy of a Galton-Watson marked tree having full mark distribution $\mu_n$ where the repeated node would be.
\end{itemize}
Since as long as we do not sample any vertex attributes from the set $S$ we will have that $\tilde T_{\emptyset(I_j)}^{(k)}(\mathbf{\tilde A}) \simeq \hat T_{\emptyset(I_j)}^{(k)}(\mathbf{\hat A})$, then for any $x_n> 0$ and $M_n$ equal to either $L_n$ for a CM or $\Lambda_n$ for an IR, we have that
\begin{align}
&\mathbb{P}_n\left( \bigcup_{j=1}^m \left\{  \hat T^{(k)}_{\emptyset(I_j)} (\mathbf{\hat A}) \not\simeq \tilde T^{(k)}_{\emptyset(I_j)} (\mathbf{\tilde A})  \right\}  \right) \notag \\
&\leq  \mathbb{P}_n\left(  \bigcup_{j=1}^m \left\{  \hat T^{(k)}_{\emptyset(I_j)} (\mathbf{\hat A}) \not\simeq \tilde T^{(k)}_{\emptyset(I_j)} (\mathbf{\tilde A})  \right\} , \, \sum_{j=1}^m \left| \hat T_{\emptyset(I_j)}^{(k)} \right| \leq m x_n, \, \sum_{i\in S} p_i \leq m x_n/M_n \right) \label{eq:BinomialBound} \\
&\hspace{5mm} +  \mathbb{P}_n\left(\left\{ \sum_{i\in S} p_i > m x_n/M_n \right\} \cup \left\{ \sum_{j=1}^m \left| \hat T_{\emptyset(I_j)}^{(k)} \right| > m x_n \right\} \right). \label{eq:MarksNodes}
\end{align}
Now note that since in the first probability $\sum_{j=1}^m\left| \hat T_{\emptyset(I_j)}^{(k)} \right| \leq mx_n$ and the chances of sampling a label from $S$ is at most $m x_n/M_n$, we have that \eqref{eq:BinomialBound} is bounded by 
$$\mathbb{P}_n( \text{Bin}(m x_n, mx_n/M_n) \geq 1) \leq \frac{m^2 x_n^2 }{M_n},$$
where $\text{Bin}(n,p)$ is a binomial random variable with parameters $(n,p)$. To analyze \eqref{eq:MarksNodes} let $Y_\mathbf{i} = p_i M_n$ if $\mathbf{\hat A}_\mathbf{i} = \mathbf{a}_i$, and note that 
$$\sum_{i\in S} p_i = \sum_{j=1}^m \sum_{\mathbf{i} \in \hat T_{\emptyset(I_j)}^{(k)}} Y_\mathbf{i}/M_n,$$
so by the union bound we obtain that \eqref{eq:MarksNodes} is bounded from above by
$$\sum_{j=1}^m \mathbb{P}_n\left( \left| \hat T_{\emptyset(I_j)}^{(k)} \right| \vee \sum_{\mathbf{i} \in \hat T_{\emptyset(I_j)}^{(k)}} Y_\mathbf{i} > x_n \right) = m \mathbb{P}_n\left( \left| \hat T^{(k)} \right| \vee \sum_{\mathbf{i} \in \hat T^{(k)}} Y_\mathbf{i} > x_n \right).$$
Now use Theorem~\ref{T.TreeToTree} to obtain that for any $\epsilon > 0$,
$$\mathbb{P}_n\left( \left| \hat T^{(k)} \right| \vee \sum_{\mathbf{i} \in \hat T^{(k)}} Y_\mathbf{i} > x_n \right) \leq P\left( \left| \mathcal{T}^{(k)} \right| \vee \sum_{\mathbf{i} \in \mathcal{T}^{(k)}} (\mathcal{Y}_\mathbf{i} + \epsilon)  > x_n \right) + o(1)$$
as $n \to \infty$, where $\mathcal{Y}_\mathbf{i}$ is equal to the second component of $\boldsymbol{X}_\mathbf{i}$. Since $|\mathcal{T}^{(k)}| < \infty$ almost surely and does not depend on $\mathscr{F}_n$, and $1/M_n = O(1/n)$ as $n \to \infty$, choosing $x_n = n^{1/2}/\log n$ proves \eqref{eq:IndependentCopies}. 

Finally, use Theorem~\ref{T.TreeToTree} applied to each of the $\{ \tilde T_{\emptyset(I_j)}^{(k)}(\mathbf{\tilde A}): 1 \leq j \leq m\}$ to obtain that there exists an i.i.d.~set $\{ \mathcal{T}^{(k)}_{\emptyset(I_j)}(\boldsymbol{A}): 1 \leq j \leq m \}$ having the same distribution as $\mathcal{T}^{(k)}(\boldsymbol{A})$, such that for any $\epsilon \in (0,1)$, 
$$\sum_{j=1}^m \mathbb{E}_n\left[ \rho(\mathbf{\tilde X}_{\emptyset(I_j)}, \boldsymbol{X}_{\emptyset(I_j)}) \right] \xrightarrow{P} 0 \quad \text{and} \quad  \mathbb{P}_n\left( \bigcap_{j=1}^m \left\{ \bigcap_{\mathbf{i} \in \mathcal{T}_{\emptyset(I_j)}^{(k)}} \{ \rho( \mathbf{\tilde X}_\mathbf{i}, \boldsymbol{X}_\mathbf{i}) \leq \epsilon \} , \, \tilde T_{\emptyset(I_j)}^{(k)} \simeq \mathcal{T}_{\emptyset(I_j)}^{(k)} \right\} \right) \xrightarrow{P} 1$$
as $n \to \infty$. This completes the proof. 
\end{proof}

\section*{Acknowledgement}
I would like to thank Sayan Banerjee for suggesting the connection between strong couplings and the notion of propagation of chaos.

\bibliographystyle{plain}
\bibliography{StrongCoupling}

\begin{thebibliography}{10}

\bibitem{aldous2007processes}
David Aldous and Russell Lyons.
\newblock Processes on unimodular random networks.
\newblock {\em Electronic Journal of Probability}, 12:1454--1508, 2007.

\bibitem{aldous2004objective}
David Aldous and J~Michael Steele.
\newblock The objective method: probabilistic combinatorial optimization and
  local weak convergence.
\newblock In {\em Probability on discrete structures}, pages 1--72. Springer,
  2004.

\bibitem{benjamini2011recurrence}
Itai Benjamini and Oded Schramm.
\newblock Recurrence of distributional limits of finite planar graphs.
\newblock In {\em Selected Works of Oded Schramm}, pages 533--545. Springer,
  2011.

\bibitem{bollobas}
B.~Bollob\'as.
\newblock A probabilistic proof of an asymptotic formula for the number of
  labelled regular graphs.
\newblock {\em European Journal of Combinatorics}, pages 311--316, 1980.

\bibitem{Bollobas2}
B.~Bollob\'as.
\newblock {\em Random graphs}.
\newblock Cambridge University Press, 2001.

\bibitem{Boll_Jan_Rio_07}
B.~Bollob\'as, S.~Janson, and O.~Riordan.
\newblock The phase transition in inhomogeneous random graphs.
\newblock {\em Random Structures \& Algorithms}, 31:3--122, 2007.

\bibitem{Brittonetal}
T.~Britton, M.~Deijfen, and A.~Martin-L\"af.
\newblock Generating simple random graphs with prescribed degree distribution.
\newblock {\em Journal of Statistical Physics}, 124:1377--1397, 2006.

\bibitem{Chain_Diez_21}
Louis-Pierre Chaintron and Anton Diez.
\newblock {\em Propagation of chaos: a review of models, methods and
  applications}.
\newblock arXiv preprint 2106.14812, 2021.

\bibitem{Chen_Lit_Olv_17}
N.~Chen, N.~Livtak, and M.~Olvera-Cravioto.
\newblock Generalized {P}age{R}ank on directed configuration networks.
\newblock {\em Random Structures \& Algorithms}, 51(2):237--274, 2017.

\bibitem{Chen_Olv_13}
N.~Chen and M.~Olvera-Cravioto.
\newblock Directed random graphs with given degree distributions.
\newblock {\em Stochastic Systems}, 3:147--186, 2013.

\bibitem{Chunglu}
F.~Chung and L.~Lu.
\newblock Connected components in random graphs with given expected degree
  sequences.
\newblock {\em Annals of Combinatorics}, 6:125--145, 2002.

\bibitem{Durrett1}
R.~Durrett.
\newblock {\em Random graph dynamics, Cambridge Series in Statistics and
  Probabilistic Mathematics}.
\newblock Cambridge University Press, 2007.

\bibitem{fra_lin_olv_21}
N.~Fraiman, T.~Lin, and M.~Olvera-Cravioto.
\newblock Stochastic recursions on directed random graphs.
\newblock {\em aXiv preprint 2010.09596}, pages 1--34, 2021.

\bibitem{Gar_vdH_Lit_19}
A.~Garavaglia, R.~van~der Hofstad, and N.~Litvak.
\newblock Local weak convergence for {P}age{R}ank.
\newblock {\em Annals of Applied Probability}, 30(1):40--79, 2020.

\bibitem{lacker2020local}
D.~Lacker, K.~Ramanan, and R.~Wu.
\newblock Local weak convergence and propagation of ergodicity for sparse
  networks of interacting processes.
\newblock {\em arXiv preprint 1904.02585}, 2020.

\bibitem{lack_ram_wu_20}
D.~Lacker, K.~Ramanan, and R.~Wu.
\newblock Marginal dynamics of interacting diffusions on unimodular
  galton-watson trees.
\newblock {\em arXiv preprint 2009.11667}, 2020.

\bibitem{Lee_Olv_20}
J.~Lee and M.~Olvera-Cravioto.
\newblock Page{R}ank on inhomogeneous random digraphs.
\newblock {\em Stochastic Processes and their Applications}, 130(4):1--57,
  2020.

\bibitem{Norros}
I.~Norros and H.~Reittu.
\newblock On a conditionally {P}oissonian graph process.
\newblock {\em Advances in Applied Probability}, 38:59--75, 2006.

\bibitem{Olvera_20}
M.~Olvera-Cravioto.
\newblock Page{R}ank's behavior under degree correlations.
\newblock {\em Ann. Applied Prob.}, 1(3):1403--1442, 2021.

\bibitem{Erdos}
P.~Erd\H os and A.~R\'enyi.
\newblock On random graphs.
\newblock {\em Publicationes Mathematicae (Debrecen)}, 6:290--297, 1959.

\bibitem{Hofstad1}
R.~van~der Hofstad.
\newblock {\em Random graphs and complex networks, Volume 1}.
\newblock Cambridge University Press, 2017.

\bibitem{Hofstad2}
R.~van~der Hofstad.
\newblock {\em Random graphs and complex networks, Volume 2}.
\newblock 2021.

\bibitem{Hofstad2005}
Remco van~der Hofstad, Gerard Hooghiemstra, and Piet Van~Mieghem.
\newblock Distances in random graphs with finite variance degrees.
\newblock {\em Random Structures \& Algorithms}, 27(1):76--123, 2005.

\bibitem{Villani_2009}
C.~Villani.
\newblock {\em Optimal transport, old and new}.
\newblock Springer, New York, 2009.

\end{thebibliography}

\end{document}